\documentclass[reqno,letterpaper,11pt]{article}

\usepackage{hyperref}

\usepackage{palatino}

\usepackage{amsmath}
\usepackage{amsfonts,amssymb,amsmath,latexsym,wasysym,mathrsfs,bbm,stmaryrd}
\usepackage[all]{xy}
\usepackage[usenames]{color}
\usepackage{epsfig}

\usepackage{graphicx}
\usepackage{subcaption}

\usepackage{amsthm}
\usepackage{thmtools}
\usepackage{thm-restate}
\declaretheorem[numberwithin=section]{theorem}
\declaretheorem[sibling=theorem]{proposition}
\declaretheorem[sibling=theorem]{definition}
\declaretheorem[sibling=theorem]{corollary}
\declaretheorem[sibling=theorem]{lemma}

\declaretheorem[sibling=theorem,style=remark]{remark}

\numberwithin{equation}{section}

\usepackage[mathcal]{euscript}
\usepackage{times}

\def\R{\mathbb R}
\def\C{\mathbb C}
\def\Z{\mathbb Z}
\def\N{\mathbb N}

\def\1{\mathbbm{1}}
\def\tensor{\otimes}

\def\bbD{\mathbb D}

\def\U{\mathbb U}

\def\GL{\mathbb{GL}}
\def\u{\mathfrak{u}}

\def\A{\mathscr{A}}
\def\G{\mathscr{G}}
\def\H{\mathscr{H}}

\def\K{\mathscr{K}}
\def\S{\mathscr{S}}
\def\SC{\mathscr{SC}}

\def\EX{\mathscr{E}}

\newcommand{\tr}{\mathrm{tr}}
\newcommand{\Tr}{\mathrm{Tr}\,}

\newcommand{\ip}[2]{\left\langle\,{#1}, {#2}\,\right\rangle} 

\long\def\symbolfootnote[#1]#2{\begingroup%
\def\thefootnote{\fnsymbol{footnote}}\footnote[#1]{#2}\endgroup}

\begin{document}

\title{The Two-Parameter Free Unitary Segal-Bargmann Transform and its Biane-Gross-Malliavin Identification}
\author{
Ching-Wei Ho\\
Department of Mathematics \\
University of California, San Diego \\
La Jolla, CA 92093-0112 \\
\texttt{cwho@ucsd.edu}
}

\date{\today} 

\maketitle

\begin{abstract}
Motivated by a conditional expectation interpretation of the Segal-Bargmann transform, we derive the integral kernel for the large-$N$ limit of the two-parameter Segal-Bargmann-Hall transform over the unitary group $\U(N)$, and explore its limiting behavior. We also extend the notion of circular systems to more general {\em elliptic systems}, in order to give an alternate construction of our new two-parameter {\em free} unitary Segal-Bargmann-Hall transform via a Biane-Gross-Malliavin type theorem.
\end{abstract}

\tableofcontents

\section{Introduction}
In early 1960s, Segal \cite{Segal1962, Segal1963} and Bargmann \cite{Bargmann1961, Bargmann1962} introduced a unitary isomorphism from $L^2$ to holomorphic $L^2$, known as the Segal-Bargmann transform (also known in the physics literature as the Bargmann transform or Coherent State transform), as a map
$$S_t : L^2(\R^n, \rho_t^n)\to L_{\text{hol}}^2(\C^n, \rho_{t/2}^{2n})$$
where $\rho_t^n$ is the standard Gaussian measure $({\scriptstyle \frac{1}{2\pi t}})^{n/2}\exp(-{\scriptstyle\frac{1}{2t}}|x|^2)\;dx$ on $\R^n$ and $L_{\text{hol}}^2(\C^n, \rho_{t/2})$ denotes the subspace of square $\rho_{t/2}$-integrable holomorphic functions on $\C^n$. $S_t$ is given by convolution with the heat kernel, followed by analytic continuation to $\C^n$.

In \cite{Hall1994}, Hall generalized the Segal-Bargmann transform to any compact Lie group $K$; the generalization is also known as the Hall's transform or the Segal-Bargmann-Hall transform. He considered the heat kernel measure $\rho_t$ of variance $t$ on $K$ determined by an $\mathrm{Ad}$-invariant inner product on $\mathrm{Lie}(K)$, and the corresponding heat kernel measure $\mu_{t/2}$ of variance $\frac{t}{2}$ on the complexification $K_\C$ of $K$; the transform, again denoted by $S_t$,  is defined as in the Euclidean case by 
$$S_tf = (e^{\frac{t}{2}\Delta_K}f)_\C$$
where $e^{\frac{t}{2}\Delta_K}$ is the time-$t$ heat operator on $K$ and $(\:\cdot\:)_\C$ means the analytic continuation from $K$ to $K_\C$. It is a unitary isomorphism between the Hilbert spaces $L^2(K, \rho_t)$ and $L_{\text{hol}}^2(K_\C, \mu_t)$. In this paper, we will be particularly interested in the case $K=\U(N)$, the group of $N\times N$ unitary matrices, and its complexification $K_\C = \GL(N)$, the general linear group of $N\times N$ invertible matrices of complex entries.

In an even more general setting, given two positive numbers $s$ and $t$ with $s > \frac{t}{2}$, Driver and Hall introduced in \cite{DriverHall1999, Hall1999} the two-parameter Segal-Bargmann transform
$$S_{s,t} : L^2(K, \rho_s)\to L_{\text{hol}}^2(K_\C, \mu_{s,t})$$
given by the same formula as $S_t$ but applied to a different domain space where $\mu_{s,t}$ is another heat kernel measure on $K_\C$. The original transform $S_t$ is the same as $S_{t,t}$. Hall also considered the case as $s\to\infty$ and $s\to\frac{t}{2}$.

Having an infinite dimensional version of the classical Segal-Bargmann transform for Euclidean spaces \cite{Segal1978}, it is natural to construct an infinite dimensional limit of the Segal-Bargmann transform for compact Lie groups. To define the transform on $\U(N)$, we fix an $Ad$-invariant inner product on the Lie algebra $\u(N) = \{X\in M(N): X^* = -X\}$ of $\U(N)$, where $M(N)$ is the space of all $N\times N$ complex matrices. The most obvious approach to the $N\to\infty$ limit would be to use an $N$-independent Hilbert-Schmidt norm on all the $\u(N)$; however, M. Gordina \cite{Gordina2000a, Gordina2000b} showed that this approach does not work because the target Hilbert space becomes undefined in the limit. Indeed, Gordina showed that with the metrics normalized in this way, in the large-$N$ limit all nonconstant holomorphic functions on $\GL(N)$ have infinite norm with respect to the heat kernel measure $\mu_t$.

Biane \cite{Biane1997b} suggested an alternative approach to the $N\to\infty$ limit of the Segal-Bargmann transform on $\U(N)$; instead of taking an $N$-independent Hilbert-Schmidt norm on all $\u(N)$, we scale the Hilbert-Schmidt norm on $\u(N)$ by an $N$-dependent constant as
\begin{equation}
\label{inNorm}
\|X\|_{\u(N)}^2 = N\Tr(X^*X) = N\sum_{j,k=1}^N|X_{jk}|^2.
\end{equation}
With this $N$-dependent constant, Biane carried out a large-$N$ limit of the Lie algebra version of the transform. He considered the classical Euclidean Segal-Bargmann transform $S_t^N$ acting on $M(N)$-valued functions with norm $\|X\|^2 = N\Tr(X^*X)$, which are given by functional calculus, componentwise on the Lie algebra $\u(N)$. The target inner product space $M(N)$ is equipped with another norm $\|X\|^2 = \frac{1}{N}\Tr(X^*X)$. Even though the result of applying $S_t^N$ to a single-variable polynomial function is in general not a single-variable polynomial function, \cite[Theorem 2]{Biane1997b} asserts that for each single-variable polynomial $P$, there is a unique single-variable polynomial $P^t$ such that
$$\lim_{N\to\infty} \| S_t^NP - P^t\|_{L^2(M(N), \gamma_{t/2}; M(N))}=0.$$
Recall that $\gamma_t$ is the variance-$t$ Gaussian measure on the Euclidean space. Biane's limit transform maps $P$ to $P^t$.

In the later sections, Biane introduced a free version of (one-parameter) Segal-Bargmann transform by means of free probability \cite{Biane1997b} as well as free stochastic calculus on a full Fock space. The underlying free probability space is the $L^2$ space of a semi-circular system whose construction is parallel to the classical construction of Gaussian variables on a Boson Fock space. The range space of the free Segal-Bargmann transform is the holomorhpic $L^2$ space of the corresponding circular system. We shall extend the free Segal-Bargmann transform to a two-parameter version whose range space is the holomorphic $L^2$ space of a two-parameter elliptic system which will be developed in Section \ref{EllipticSystems}. In the other direction of generalization, Kemp \cite{Kemp2005} studied the generalization of the free Segal-Bargmann transform on different Fock spaces.

Biane also constructed the ``large-$N$ limit Segal-Bargmann transform on $\U(N)$" $\G_t$ using Malliavin calculus techniques and gave a Gross-Malliavin identification \cite{GrossMalliavin1996}. $\G_t$ is a unitary isomorphism between an $L^2$ space of a measure on the unit circle $\U$, which is the multiplicative analogue of the semi-circular distribution, and a reproducing kernel Hilbert space of analytic functions on some domain of the complex plane. In his paper \cite{Biane1997b}, Biane did not prove the transform $\G_t$ can be obtained by taking a large-$N$ limit from applying the Segal-Bargmann transform $S_t^N$ on $\U(N)$ to $M(N)$-valued functions; instead, he used the Gross-Malliavin type approach to identify the polynomials that should be transformed to monomials, and computed their generating function.

The construction of the unitary isomorphism  $\G_t$ uses the tools developed in \cite{Biane1998}. Motivated from development of free processes with free additive/multiplicative increments, Biane showed the existence of (free) Markov transition functions for such processes. He then applied the result to the free unitary Brownian motion in \cite{Biane1997b} to obtain a kernel which gives a free analogue of constructing the heat kernel from the Gaussian measure. The transform $\G_t$ is defined by integrating the $L^2$ function against the free kernel. 

Driver, Hall and Kemp \cite{DriverHallKemp2013} proved, by explicit calculation with combinatorial tools and solving partial differential equations, that $\G_t$ is the direct limit of the Segal-Bargmann transform on $\U(N)$. They considered the two-parameter Segal-Bargmann transform $S_{s,t}^N$ on $\U(N)$ acting on $M(N)$ - valued functions and showed that for each single-variable polynomial $P$, even though $S_t^N P$ is typically not a single-variable polynomial, the sequence $(S_{s,t}^N P)_{N=1}^\infty$ does have a single-variable polynomial limit $\G_{s,t}\,p$ in the following sense:
$$\lim_{N\to\infty}\|S_{s,t}^NP - \G_{s,t}P\|_{L^2(\GL(N), \mu_{s,t}^N; M(N))} = 0.$$
They computed in \cite[Theorem 1.31]{DriverHallKemp2013} that for $s>\frac{t}{2}>0$, there are polynomials $p_{s,t}^{(k)}$, $k=1,2,\cdots$, such that $\G_{s,t}p_{s,t}^{(k)} = (\cdot)^k$ and the generating function $\Pi_{s,t}(u,z) = \sum_{k\geq 1}p_{s,t}^{(k)}(u)z^k$
satisfies
\begin{equation}
\label{GenFun}
\Pi_{s,t}(u,ze^{\frac{1}{2}(s-t)\frac{1+z}{1-z}}) = \left(1-uze^{\frac{s}{2}\frac{1+z}{1-z}}\right)^{-1} - 1;
\end{equation}
in particular, $\G_{t, t} = \G_t$, since the generating function in $s=t$ case defined in \cite{Biane1997b} concerning the polynomials whose transform under $\G_t$ are monomials. The main results of \cite{DriverHallKemp2013} were also proved simultaneously and independently by C\'{e}bron in \cite{Cebron2013}, using free probability and combinatorics techniques. The techniques \cite{DriverHallKemp2013} and \cite{Cebron2013} used were different; C\'ebron considered Brownian motions on $\U(N)$ and $\GL(N)$ and the free Brownian motions while Driver, Hall and Kemp did not use free probability at all. However, C\'{e}bron related the one-parameter free unitary Segal-Bargmann transform of polynomials to computing conditional expectations. We will combine C\'{e}bron's and Driver, Hall and Kemp's work to give a conditional expectation form of the two-parameter free unitary Segal-Bargmann transform.

Gross and Malliavin \cite{GrossMalliavin1996} showed that the Segal-Bargmann transform on compact Lie groups can be recovered from an infinite dimensional version of the Segal-Bargmann transform through the endpoint evaluation map, by imbedding the $L^2$ space of the heat kernel measure on a Lie group of compact type into the $L^2$ space of the Wiener measure associated with Brownian motion on its Lie algebra. Biane \cite{Biane1997b} used a free analogue of the Gross-Malliavin identification to define the putative large-$N$ limit of the Segal-Bargmann transform on $\U(N)$. It can be recovered from the free Segal-Bargmann transform through functional calculus, by imbedding the $L^2$ space of the measure on the unit circle into the $L^2$ space of the free unitary Brownian motion on a free probability space. The free multiplicative Brownian motion was generalized to the free multiplicative $(s,t)$-Brownian motion which has been studied for couples of years (see, e.g., \cite{CebronKemp2015, Kemp2015}); it satisfies a free stochastic differential equation which looks the same as the stochastic differential equation for $(s,t)$-Brownian motion on $\GL(N)$. The two-parameter free Brownian motion is the main ingredient of the range space of the two-parameter free Segal-Bargmann transform which will be discussed in Section \ref{EllipticSystems} of the present paper.

The rest of this introduction is devoted to summary and explanation of the results of the current paper. We consider the family of distributions $\nu_t$ of a free unitary Brownian motion at time $t$ on the unit circle $\U$ whose $\Sigma$-transform $\frac{f_t(z)}{z}$ is $e^{\frac{t}{2}\frac{1+z}{1-z}}$ (See Section \ref{FreeProb} for definition). We denote the inverse of $f_t$ by $\chi_t$, which is analytic on the unit disk $\bbD$ (See \cite{Biane1997b}). 

We first put \cite{Cebron2013} and \cite{DriverHallKemp2013} together to give the following proposition; see Section \ref{ConditionalExpectationSection}.
\begin{proposition}
\label{stCondExp}
Let $b_{s,t}$ be the free multiplicative $(s,t)$-Brownian motion and $u_t$ be the free unitary Brownian motion. Suppose that the processes $u_t$ and $b_{s,t}$ are free to each other. Then we have, with an abused notation $b_{s,t} = b_{s,t}(1)$,
\begin{equation*}
\G_{s,t} f(b_{s,t}) = \tau[f(b_{s,t}u_t) | b_{s,t}]
\end{equation*}
 for all Laurent polynomials $f$.
\end{proposition}

We let $u_t$ and $\tilde{u}_t$ be free unitary Brownian motions which are free to each other. For $s>t$, the operator $b_{s,t}(1)$ has the same holomorphic moments as $u_{s-t}$, which is stated in \cite{Kemp2015}. Theorem \ref{FreeConv}, which was proved in \cite{Biane1998} by Biane, asserts, with $\mu = \nu_{s-t}$ and $\nu = \nu_t$ where $\nu_t$ is the distribution of $u_t$, that the existence of a Feller Markov kernel $H = h(\cdot, d\omega)$ on $\U\times\U$ such that
$$\tau[f(u_{s-t}\tilde{u}_t) | u_{s-t}] = Hf(u_{s-t})$$
for any bounded Borel function $f$ and kernel $h(\zeta,d\omega)$ is determined by the moment generating function
$$\int_\U \frac{z\omega}{1-z\omega}\:h(\zeta,d\omega) = \frac{\chi_{s,t}(z)\zeta}{1-\chi_{s,t}(z)\zeta}$$
where $\chi_{s,t} = f_{s-t}\circ\chi_s$ is an analytic function on $\bbD$. It follows that in the $s>t$ case, by Proposition \ref{stCondExp}, again since $u_{s-t}$ and $b_{s,t}(1)$ have the same holomorphic moments,
\begin{equation*}
\G_{s,t}f(u_{s-t}) = \tau(f(u_{s-t}\tilde{u}_t | u_{s-t}) = Hf(u_{s-t}) = \int_\U f(\omega)\:h(u_{s-t},d\omega)
\end{equation*}
which is an integral transform version of the two-parameter Segal-Bargmann transform. We will also construct such a kernel for $s\geq\frac{t}{2}>0$. The computation above will be given in much details in Section \ref{ConditionalExpectationSection}.

Having given this motivation, we then move on to establish the integral formula for the two-parameter free unitary Segal-Bargmann transform. We are concerned with $s\geq \frac{t}{2}>0$. We first prove that $\chi_{s,t} = f_{s-t}\circ \chi_s$ is an injective conformal map from $\bbD$ onto its image (see Definition \ref{CHI_ST}). Then we define a kernel $k_{s,t}(\:\cdot\:, d\omega)$ whose $L^2$ space is the same as $L^2(\nu_s)$; for the exact statement, see Theorem \ref{Nu s,t} and Proposition \ref{AbsCont}. And the integral formula for the large-$N$ limit of the Segal-Bargmann transform on $\U(N)$, called the free unitary Segal-Bargmann transform, is a unitary isomorphism from $L^2(\nu_s)$ to a reproducing kernel Hilbert space $\A_{s,t}$ defined for each $f\in L^2(\nu_s)$,
$$\tilde{\G}_{s,t}f(\zeta) = \int_\U f(\omega)\frac{|1-\chi_{s,t}(\omega)|^2}{(\zeta-\chi_{s,t}(\omega))(\zeta^{-1}-\bar{\chi}_{s,t}(\omega))}\;\frac{1-|\chi_{s,t}(\omega)|^2}{1-|\chi_s(\omega)|^2}\nu_{s}(d\omega)$$
for all $\zeta\in \Sigma_{s,t}$ where the domain $\Sigma_{s,t}$ of the analytic function $\G_{s,t}f$  has a very precise description given in Section \ref{TheTwoParaHeatKernel}. The topology of $\Sigma_{s,t}$ depends on $s$ only. For $s<4$, $\Sigma_{s,t}$ is simply connected while for $s>4$, $\Sigma_{s,t}$ is of conformal type as an annulus; in the complicated case $s=4$, $\Sigma_{s,t}$ itself is simply connected but the complement of $\bar{\Sigma}_{s,t}$ has two components. Here is a theorem which summarizes Theorem \ref{IntMainThm} and Theorem \ref{IntCoincides}:
\begin{theorem}
\begin{enumerate}
\item The transform $\tilde{\G}_{s,t}$ is a unitary isomorphism between the Hilbert spaces $L^2(\nu_s)$ and the reproducing kernel Hilbert space $\mathscr{A}_{s,t}$ of analytic functions on $\Sigma_{s,t}$ generated by the positive-definite sesqui-analytic kernel
$$K(z,\zeta)=\int_\U\frac{|1-\chi_{s,t}(\omega)|^2}{(z-\chi_{s,t}(\omega))(z^{-1}-\bar{\chi}_{s,t}(\omega))}\frac{|1-\chi_{s,t}(\omega)|^2}{(\zeta-\chi_{s,t}(\omega))(\zeta^{-1}-\bar{\chi}_{s,t}(\omega))}\;\left(\frac{1-|\chi_{s,t}(\omega)|^2}{1-|\chi_s(\omega)|^2}\right)^2\nu_{s}(d\omega).$$
\item $\tilde{\G}_{s,t}$ coincides to $\G_{s,t}$, the large $N$-limit of the Segal-Bargmann transform on $\U(N)$; i.e. $\tilde{\G}_{s,t}$ extends $\G_{s,t}$ to a unitary isomorphism between the two Hilbert spaces.
\end{enumerate}
\end{theorem}
We also compute the limit behavior of the domain $\Sigma_{s,t}$. If we hold $t$ fixed and let $s\to\infty$, the region $\Sigma_{s,t}$ converges to an annulus with inner and outer radii $e^{-\frac{t}{2}}$ and $e^{\frac{t}{2}}$; in the $s=t$ case, if we let $s=t\to\infty$, $\Sigma_{t,t}$ is asymptotically an annulus of inner and outer radii $e^{-\frac{t}{2}}$ and $e^{\frac{t}{2}}$ respectively.

We will also define the two-parameter analogue of the free Segal-Bargmann transform on some free probability spaces and prove the Biane-Gross-Malliavin identification in the two-parameter setting. The $L^2$ completion of a semicircular system is a free analogue of the $L^2$ space of the Gaussian measure in the classical case and the $L^2$ completion of the holomorphic elliptic system is the analogue of the holomorphic $L^2$ space of a certain anisotropic Gaussian measure, cf. \cite{DriverHall1999}. The two-parameter free Segal-Bargmann transform is a unitary isomorphism between the two free probability spaces. The Biane-Gross-Malliavin identification is the commuting diagram between free probability spaces: the $L^2$ spaces of free unitary Brownian motion, free $(s,t)$-Brownian motion and the $L^2$ function spaces of the integral transform. For $s>\frac{t}{2}>0$, the integral transform of the free unitary Segal-Bargmann transform $\G_{s,t}$ can be recovered from the free Segal-Bargmann transform $\S_{s,t}$ through functional calculus on the free probability spaces. The identification is presented in the following theorem, a full statement is stated as Theorem \ref{GrossMalliavin}. 
\begin{theorem}
Let $s>\frac{t}{2}>0$. Suppose $u_s(r)$ is a time-rescaled free unitary Brownian motion given by the (unique) solution of the free stochastic differential equation
$$du_s(r) = i\sqrt{s}u_s(r)dx_r + \frac{s}{2}u_s(r)dr.$$
We abuse the notations to write $u_s(1)$ and $b_{s,t}(1)$ as the functional calculus and holomorphic functional calculus respectively. Then the following diagram of Segal-Bargmann transforms and functional calculus commute:
\begin{displaymath}
    \xymatrix{
        L^2(\nu_s) \ar[rr]^{u_{s}(1)} \ar[d]_{\G_{s,t}}  && L^2(u_s(1),\tau) \ar[d]^{\S_{s,t}} \\
       \A_{s,t}  \ar[rr]_{b_{s,t}(1)}        && L_{\text{hol}}^2(b_{s,t}(1),\tau).
       }
\end{displaymath}
All maps are unitary isomorphisms.
\end{theorem}

The paper is organized as follows. In section \ref{Preliminary}, we provide definitions, background and main tools for this paper. In section \ref{ConditionalExpectationSection}, we will explain how we can combine \cite{Cebron2013} and \cite{DriverHallKemp2013} to give the two-parameter free unitary Segal-Bargmann transform in the form of conditional expectation. We will then make use of the result to obtain a simplified version of an integral version of the two-parameter free unitary Segal-Bargmann transform. In section \ref{IntTrans}, we derive the integral representation for the two-parameter free unitary Segal-Bargmann transform, which is the large-$N$ limit of the Segal-Bargmann-Hall transform on $\U(N)$, through a direct generalization of the work in the previous section. In section \ref{BianeGrossMalliavin}, we first introduce elliptic systems which extends circular systems and define the two-parameter free Segal-Bargmann transform; we then prove a version of the analogue of the Biane-Gross-Malliavin theorem which recovers the free unitary Segal-Bargmann transform from the free Segal-Bargmann transform by means of free stochastic calculus and funcitonal calculus.

\section{Background and Preliminaries}
\label{Preliminary}
\subsection{Heat Kernel Analysis on $\U(N)$}
In this section, we give the main lines of how to construct the Laplacian on $\U(N)$ and the definition of the two-parameter Segal-Bargmann transform on $\U(N)$. The $N$-dimensional unitary group $\U(N)$ is a compact matrix Lie group with Lie algebra $\u(N) = \{X\in M(N): X^* = -X\}$. The Lie algebra $\u(N)$ is equipped with the scaled Hilbert-Schmidt (real) inner product
\begin{equation}
\label{NormOnu(N)}
\ip{X}{Y}_{\u(N)} = -N\Tr(XY).
\end{equation}
\begin{definition}
For each $X\in\u(N)$, the associated left-invariant vector field in the direction $X$ is the differential operator $\partial_X: C^\infty (\U(N), M(N))\to C^\infty(\U(N), M(N))$ given by
$$(\partial_X F)(A)=\left.\frac{d}{dt}\right|_{t=0}F(Ae^{tX})$$
for all $A\in \U(N)$ whenever $F\in C^\infty (\U(N), M(N))$.
\end{definition}
\begin{remark}
Most authors refer to $\partial_X$ as $\tilde{X}$.
\end{remark}
\begin{definition}
Let $\beta_N$ be an orthonormal basis for $\u(N)$ under the inner product given in \eqref{NormOnu(N)}. The Laplacian $\Delta_{\U(N)}$ on $C^\infty(\U(N), M(N))$ is the operator
$$\Delta_{\U(N)} = \sum_{X\in\beta_N} \partial_X^2$$
which is independent of the choice of the orthonormal basis $\beta_N$. For $t>0$ the heat operator is $e^{\frac{t}{2}\Delta_{\U(N)}}$ and the heat kernel measure $\rho_t^N$ is characterized as the linear functional
$$\int_{\U(N)} f(U)\; \rho_t^N(dU) = \left(e^{\frac{t}{2}\Delta_{\U(N)}}f\right)(I_N)$$
for all $f\in C(\U(N))$ where $I_N$ is the identity matrix in $M(N)$. \\

The Lie group complexification of $\U(N)$ is $\GL(N)$; in particular $\mathfrak{gl}(N, \C) = \u(N)\oplus i\u(N)$. We define the Laplacian $\Delta_{\GL(N)}$ to be
$$\Delta_{\GL(N)} = \sum_{X\in\beta_N}\partial_X^2+\sum_{X\in\beta_N}\partial_{iX}^2.$$

Let $s>\frac{t}{2}>0$. We define the operator $A_{s,t}^N$ on $C^\infty(\GL(N), M(N))$ by
$$A_{s,t}^N = \left(s-\frac{t}{2}\right)\sum_{X\in\beta_N}\partial_X^2+\frac{t}{2}\sum_{X\in\beta_N}\partial_{iX}^2.$$
The measure $\mu_{s,t}^N$ on $\GL(N)$ is determined by
$$\int_{\GL(N)} f(A)\; \mu_{s,t}^N(dA) = \left(e^{\frac{1}{2}A_{s,t}^N}f\right)(I_N)$$
for all $f\in C_c(\GL(N))$.
\end{definition}

Observe that $A_{s,s}^N = \frac{s}{2}\Delta_{\GL(N)}$ and $A_{s,0} = s \Delta_{\U(N)}$; $A_{s,t}^N$ interpolates between the two heat kernels.

We now give the definition of the scalar unitary Segal-Bargmann transform and boosted unitary Segal-Bargmann transform on $\U(N)$. The space $M(N)$ is equipped with the inner product
$$\ip{A}{B}_{M(N)} = \frac{1}{N}\Tr(B^*A)=\frac{1}{N}\sum_{j,k=1}^N A_{jk}\bar{B}_{jk}.$$

\begin{definition}
Let $s>\frac{t}{2}>0$. The scalar unitary Segal-Bargmann transform
$$S_{s,t}^N: L^2(\U(N), \rho_s^N)\to L_{\text{hol}}^2(\GL(N), \mu_{s,t}^N)$$
is defined by the analytic continuation of $e^{\frac{t}{2}\Delta_{\U(N)}}f$ to an entire function on $\GL(N)$; it is a unitary isomorphism between the two Hilbert spaces .We note that the function $e^{\frac{t}{2}\Delta_{\U(N)}}f$ always possesses an analytic continuation to entire $\GL(N)$ (see \cite{Driver1995, DriverHall1999, HallSengupta1998}).

The transform $S_{s,t}^N$ also acts on $M(N)$-valued functions componentwise; we abuse the notation to define
$$S_{s,t}^N : L^2(\U(N), \rho_s^N)\otimes M(N) \to L_{\text{hol}}^2(\GL(N), \mu_{s,t}^N)\otimes M(N)$$
which is also an unitary isomorphism. All tensor products are over $\C$.
\end{definition}

We will discuss the action of the boosted Segal-Bargmann transform throughout the paper; from this point on, $S_{s,t}^N$ will always refer to the boosted unitary Segal-Bargmann on $\U(N)$. Studying the large-$N$ limit of $S_{s,t}^N$ on single-variable polynomials helps understand the large-$N$ limit of the operator on functions given by functional calculus. In general, for a single-variable polynomial $p$, $S_{s,t}^N \,p$ is not necessarily a single-variable polynomial; an example from \cite{DriverHallKemp2013} is that, if we take $p(u) = u^2$, then
\begin{align*}
(S_{s,t}^N\:p)(Z) &= e^{-t}\left[\cosh(t/N)Z^2-t\frac{\sinh(t/N)}{t/N}Z\tr Z\right]\\
&= e^{-t}[Z^2-tZ\tr Z]+O\left(\frac{1}{N^2}\right)
\end{align*}
in which $\tr Z$ is involved. However, we have  \cite[Theorem 1.30]{DriverHallKemp2013}:
\begin{theorem}
Let $s>\frac{t}{2}>0$, and $p$ be a single-variable Laurent polynomial. Then there is a unique single-variable polynomial $\G_{s,t}\,p$ such that
$$\|S_{s,t}^N\, p - \G_{s,t}\, p\|_{L^2(\GL(N), \mu_{s,t}^N; M(N))}^2 = O\left(\frac{1}{N^2}\right).$$
\end{theorem}
The limit transform $\G_{s,t}$ is referred as the free unitary Segal-Bargmann transform. The following theorem \cite[Theorem1.31]{DriverHallKemp2013} showed, restricting to the space of all single-variable polynomials, $\G_{s,t}$ coincides to the integral transform defined on an $L^2$ space of a measure on the unit circle introduced by Biane in \cite{Biane1997b}.
\begin{theorem}
Let $s>\frac{t}{2}>0$ and let, for $k\geq 1$, $p_{s,t}^{(k)}$ be the polynomials such that $\G_{s,t} p_{s,t}^{(k)}$ is the single-variable monomial of order $k$. Then the power series
$$\Pi_{s,t}(u,z) = \sum_{k\geq 1} p_{s,t}^{(k)}(u)z^k$$
converges for all sufficiently small $|u|$ and $|z|$, and the generating function $\Pi_{s,t}$ satisfies \eqref{GenFun}.
\end{theorem}
As a last remark in this section, Section \ref{IntegralTransform} will concern the construction of the integral transform formula of $\G_{s,t}$.

\subsection{Free Probability}
\label{FreeProb}
\begin{definition}
\begin{enumerate}
\item We call $(\A, \tau)$ a $W^*$-probability space if $\A$ is a von Neumann algebra and $\tau$ is a normal, faithful tracial state on $\A$. The elements in $\A$ are called (noncommuntative) random variables. 
\item The $\ast$ - subalgebras $A_1, \cdots A_n\subseteq \A$ are called free or freely independent if, given any $i_1, i_2,\cdots, i_m\in\{1,\cdots,n\}$ with $i_k\not= i_{k+1}$, $a_{i_j}\in\A_{i_j}$ are centered, then we also have $\tau(a_{i_1}a_{i_2}\cdots a_{i_m})=0$. The random variables $a_1,\cdots, a_m$ are free or freely independent if the $\ast$-subalgebras they generate are free.
\item For a self-adjoint element $a\in\A$, the law $\mu$ of $a$ is a compactly supported probability measure on $\R$ such that whenever $f$ is a continuous function, we have
$$\int_\R f\;d\mu = \tau(f(a)).$$
\end{enumerate}
\end{definition}

\begin{definition}[$\Sigma$-transform]
\label{SigmaTrans}
Let $\mu$ be a probability measure on $\C$. Define the function 
$$\psi_\mu(z)=\int_{\C}\frac{\omega z}{1-\omega z}\mu(d\omega)$$
for those $z$ with $\frac{1}{z}\not\in\text{supp}\:\mu$. $\psi_\mu$ is analytic on its domain. If $\mu$ is supported in $\U$, it is customary to restrict $\psi_\mu$ to the unit disk $\bbD$; if $\mu$ is supported in $\R$, it is customary to restrict $\psi_\mu$ to the upper half-plane $\C_+$. Define $\chi_\mu = \psi_\mu/(1+\psi_\mu)$. This function is injective on a neighborhood of $0$ if $\text{supp}\:\mu\subseteq \U$ and the first moment of $\mu$ is nonzero; it is injective on the left-half plane $i\C_+$ if $\text{supp}\:\mu\subseteq \R_+$, cf. \cite{BercoviciVoiculesc1993}. The $\Sigma$-transform $\Sigma_\mu$ is the analytic function
$$\Sigma_\mu(z) = \frac{\chi_\mu^{-1}(z)}{z}$$
for $z$ in a neighborhood of $0$ in the $\U$-case and for $z\in\chi_\mu(i\C_+)$ in the $\R_+$-case.
\end{definition}
\begin{remark}
The function $\chi_\mu$ defined here is usually denoted by $\eta_\mu$ and is called the $\eta$-transform of the measure $\mu$. We choose to use the notation $\chi$ here because in the rest of the paper, we follow the notation of Biane in \cite{Biane1997b}.
\end{remark}

A measure on the unit circle $\U$ is completely determined by its moments; the $\eta$ and $\Sigma$ -  transforms characterize the measures on $\U$ by the corresponding class of holomorphic functions on $\bbD$. The corresponding class of holomorphic functions for the $\eta$-transform is those analytic self maps $f$ on $\bbD$ satisfying $|f(z)|\leq |z|$, cf \cite{BellinschiTeodorThesis}; the class of functions for the $\Sigma$-transform can be easily seen from the definition of the $\Sigma$-transform which is related to the $\eta$-transform.

For two freely independent unitary random variables $x$ and $y$ with laws $\mu$ and $\nu$ respectively. We define the free multiplicative convolution $\mu\boxtimes \nu$ to be the law of the unitary random variable $xy$. $\Sigma$-transform plays an important role to analyze the free multiplicative convolution; it makes the free multiplicative convolution multiplicative in the following sense:
$$\Sigma_{\mu\boxtimes \nu} = \Sigma_\mu \Sigma_\nu.$$

Consider measures $\{\nu_t\}_{t\in\R}$ supported on $\U$ for $t\geq 0$ and supported on $\R$ for $t\leq 0$ having $\Sigma$-transforms
$$\Sigma_{\nu_t}(z) = e^{\frac{t}{2}\frac{1+z}{1-z}}.$$
Write $f_t(z) = z\Sigma_{\nu_t}(z)$ for all $t\in \R$, which is a meromorphic function on $\C$ with the only singularity at $1$. The following proposition summarizes the results in \cite{BelinschiBercovici2004, BelinschiBercovici2005, BercoviciVoiculesc1992, Biane1997b, Zhong2014} concerning the maps $f_t$.

\begin{proposition}
\label{PropOfMaps}
For $t>0$, $\nu_t$ has a continuous density $\rho_t$ with respect to the normalized Haar measure on $\U$, the unit circle. For $0<t<4$, its support is the connected arc
$$\text{supp}\:\nu_t=\left\{e^{i\theta} : -\frac{1}{2}\sqrt{t(4-t)}-\arccos\left(1-\frac{t}{2}\right)\leq \theta\leq \frac{1}{2}\sqrt{t(4-t)}+\arccos\left(1-\frac{t}{2}\right)\right\}$$
while $\text{supp}\:\nu_t = \U$ for $t\geq 4$. The density $\rho_t$ is real analytic on the interior of the arc. It is symmetric about $1$, and is determined by $\rho_t(e^{i\theta})=\text{Re}\:\kappa_t(e^{i\theta})$ where $z = \kappa_t(e^{i\theta})$ is the unique solution (with positive real part) to 
$$\frac{z-1}{z+1}e^{\frac{t}{2}z}=e^{i\theta}.$$
The function $f_t$ maps $\Omega_t = \{z\in\bbD: f_t(z)\in\bbD\}$ onto $\bbD$ conformally and extends to a homeomorphism from $\bar{\Omega}_t$ to $\bar{\bbD}$. $\Omega_t$ is a Jordan domain and 
$$\left(\frac{1+\,\cdot}{1-\,\cdot}\right)(\Omega_t)=\left\{x+iy: x>0, \left|\frac{x-1}{x+1}e^{\frac{t}{2}x}\right|<1, |y|<\sqrt{\frac{(x+1)^2-(x-1)^2e^{tx}}{e^{tx}-1}}\right\}.$$

For $t<0$, $\nu_t$ has a continuous density $\rho_t$ with respect to Lebesgue measure on $\R_+$. The support is the connected interval $\text{supp} \nu_t = (x_+(t), x_-(t))$ where
$$x_{\pm}(t) = \frac{2-t\pm\sqrt{t(t-4)}}{2}e^{-\frac{1}{2}\sqrt{t(t-4)}}.$$
The density $\rho_t$ is real analytic on the interval $(x_-(t), x_+(t))$ unimodal with peak at its mean $1$; it is determined by $\rho_t(x) = \frac{1}{\pi x}\text{Im}\:l(x)$ where $z=l(x)$ is the unique solution to
$$\frac{z}{z-1}e^{-t(z-\frac{1}{2})}=x.$$
The function $f_t$ maps $\Omega_t = \{re^{i\theta}\in \C_+: 0<r<\infty, \gamma_t(r)<\theta<\pi\}$ where $\gamma_t(r)$ satisfies 
$$\frac{\sin \gamma_t(r)}{\gamma_t(r)}\frac{r}{1+r^2-2r\cos(\gamma_t(r))}=-\frac{1}{t}$$
onto $\C_+$ conformally and extends to a homeomorphism from $\bar{\Omega}_t$ to $\bar{\C}_+$. $\Omega_t$ is a Jordan domain and $\gamma_t(r)$ is a strictly increasing function of $r$ on the interval $(z_-(t), 1]$ and a strictly decreasing function of $r$ on $[1, z_+(t))$ where
$$z_{\pm}(t) = \frac{2+t\pm\sqrt{t(t+4)}}{2}.$$
\end{proposition}

\begin{remark}
\label{IntersectImag}
When $4>t>0$, Biane also showed in \cite{Biane1997b} that $\overline{\left(\frac{1+\,\cdot}{1-\,\cdot}\right)(\Omega_t)}\cap i\R = i\left[-\sqrt{\frac{4}{s}-1}, \sqrt{\frac{4}{s}-1}\,\right]$. Therefore, $\bar{\Omega}_t\cap \U$ is the arc 
$$\left[\frac{-i\sqrt{\frac{4}{s}-1}-1}{-i\sqrt{\frac{4}{s}-1}+1}, \frac{i\sqrt{\frac{4}{s}-1}-1}{i\sqrt{\frac{4}{s}-1}+1}\right]$$
which does not include $1$.
\end{remark}
\begin{remark}
\label{Density}
In \cite[Proposition 10]{Biane1997b}, Biane proved that by the Herglotz Representation Theorem
$$\nu_t(d\omega) = \text{Re}\:\left(\frac{1+\chi_t(\omega)}{1-\chi_t(\omega)}\right)\:d\omega$$
where $\chi_t = \chi_{\nu_t}$ is defined on $\bbD$ and extended to a homeomorphism on $\bar{\bbD}$.
\end{remark}
\begin{remark}
As mentioned in \cite{Biane1997b}, the function $f_t$ preserves inversion, for all $t>0$; we can extend $\chi_t$ to $\C\setminus\bar{\bbD}$ so that $\chi_t$ and $f_t$ are still inverse to each other. If $0<t<4$, $\chi_t$ can be analytically continued to the complement of $\text{supp}\:\nu_t$ in the Riemann sphere $\C_\infty$. For $t\geq 4$, $\chi_t$ can be extended to $\C_\infty\setminus\text{supp}\:\nu_t = \bbD\cup (\C_\infty\setminus\bar{\bbD})$. $\left.\chi_{t}\right|_\bbD$ and $\left.\chi_t\right|_{\C_\infty\setminus\bar{\bbD}}$ just differ by an inversion. If we put $\Sigma_t=\C_\infty\setminus\overline{\chi_t(\C_\infty\setminus\text{supp}\:\nu_t)}$, the range of $\G_t$ lies inside the Hardy space $H^2(\Sigma_t)$, equipped with different inner products.
\end{remark}
\begin{remark}
\label{RealSol}
For the $t<0$ case, since $\gamma_t(r)$ is a strictly increasing function of $r$ on the interval $(z_-(t), 1]$ and a strictly decreasing function of $r$ on $[1, z_+(t))$. we have for each $\theta\in[0,\gamma_t(1))$, the quadratic equation $r^2-\left(2\cos \gamma_t(r)+\frac{t\sin \gamma_t(r)}{\gamma_t(r)}\right)r+1=0$ has two nonnegative roots, one $<1$ and the other $>1$. So if $r<1$, $r$ is a strictly increasing function of $\theta$ for $\theta\in[0,\gamma_t(1)]$.
\end{remark}

We will continue using the notations $\chi_t, f_t, \nu_t, x_\pm(t), z_\pm(t), \Omega_t$, $\gamma_t$ throughout the paper.\\

We try to make sense of the free convolution of a function and a measure. We first state a theorem which was first proved in \cite{Biane1998}.
\begin{theorem}
\label{FreeConv}
Let $(A,\tau)$ be a $W^*$-probability space, $B\subseteq A$ be a von Neumann subalgebra, and $U, V\in A$ such that $U$ and $V$ are unitary, with distributions $\mu$ and $\nu$ respectively. Suppose that $U\in B$ and $V$ is free with $B$. Then there exists a Feller Markov kernel $\K=k(\zeta, d\omega)$ on $\U\times\U$ and an analytic function $F$ defined on $\bbD$ such that
\begin{enumerate}
\item for any bounded Borel function $f$ on $\U$, $$\tau(f(UV)|B) = \K f(U);$$
\item $F(z)\leq |z|$, for all $z\in \bbD$;
\item for all $z\in\bbD$, $$\int_\U \frac{z\omega}{1-z\omega}\: k(\zeta, d\omega) = \frac{F(z)\zeta}{1-F(z)\zeta};$$
\item for all $z\in\bbD$, $\psi_\mu(F(z)) = \psi_{\mu\boxtimes \nu}(z)$.
\end{enumerate}
If $\mu$ has nonzero first moment, the map $F$ is uniquely determined by (2) and (4).
\end{theorem}
The $F$ in Theorem \ref{FreeConv} is called the subordination function of $\psi_{\nu_{\mu\boxtimes\nu}}$ with respect to $\psi_\mu$. In the classical case, we can construct from a measure a Feller Markov kernel by means of convolution. Biane \cite{Biane1997b, Biane1998} suggested that, given a measure $\nu$ on $\U$ and a bounded Borel function $f$,  with $\mu=\delta_1$, $\K f$ is the free convolution of a function and a measure. The choice of $\mu=\delta_1$ can be compared to the kernel constructed from (additive) convolution that the kernel at $0$ is simply the original measure.

\subsection{Semi-circular System on a Fock Space}
\label{SemiCircularSection}
We denote $(\A,\tau)$ a $W^*$-probability space and $\H$ a real Hilbert space, with inner product $\ip{\cdot}{\cdot}$. We first recall the definition of semi-circular system.
\begin{definition}
A linear map $s:\H\to \A$ is a a semi-circular system if 
\begin{enumerate}
\item for each $h\in\H$, $s(h)$ is self-adjoint and has semi-circular distribution of variance $\ip{h}{h}$;
\item whenever $h_1,\cdots,h_n\in\H$ are orthogonal, the family $(s(h_j))_{j=1}^n$ is free.
\end{enumerate}
\end{definition}

We shall construct the semi-circular systems on a free Fock space. The semi-circular system was constructed in \cite{Biane1997b}. Let $\H^\C$ be the complexification of $\H$. Denote $F(\H^\C)$ the free Fock space associated to $\H^\C$, which is the Hilbert space orthogonal direct sum
$$F(\H^\C) = \C\:\Omega\oplus\bigoplus_{n=1}^\infty (\H^\C)^{\tensor n}$$
where $\Omega$ is a unit vector orthogonal to $\H$, called the vacuum. For each $h\in\H$, we define the annihilation and creation operators $a_h$ and $a_h^*$, which are bounded operators on $F(\H^\C)$ and adjoint to each other, by the linear extension of
\begin{align*}
a_h(\Omega) = &\:0,\\
a_h(h_1\tensor\cdots\tensor h_n) = &\:\ip{h_1}{h}h_2\tensor\cdots\tensor h_n,\\
a_h^*(h_1\tensor\cdots\tensor h_n) = &\:h\tensor h_1\tensor\cdots\tensor h_n.\\
\end{align*}
Note that the convention for inner product here is linear in the first entry and sesqui-linear in the second entry. Obviously $a_fa_g^* = \ip{g}{f}$ for all $f,g\in \H$. Therefore any product of creation and annihilation operators is a scalar multiple of 
$$a_{f_1}^*\cdots a_{f_n}^*a_{g_1}\cdots a_{g_m}.$$
For each $h\in\H$, we define
$$X(h) = a_h+a_h^*$$
and let $\SC(\H) = W^*\{X(h): h\in\H\}$ be the von Neumann subalgebra of the operators $\mathscr{B}(F(\H^\C))$ on $F(\H^\C)$ generated by all $X(h)$ with $h\in\H$. We also let $\tau$ to be the restriction to $\SC(\H)$ of the pure state associated to the vector $\Omega$, i.e. $\tau(T) = \ip{T\Omega}{\Omega}$ for $T\in\SC(\H)$.

We now quote a proposition from \cite[Proposition 2]{Biane1997b} whose proof can be derived from \cite{VoiculescuDykemaNica1992}:
\begin{proposition}
The state $\tau$ is a faithful normal tracial state on $\SC(\H)$ so that $(\SC(\H),\tau)$ is a $W^*$-probability space. Moreover, the map $X:\H\to (\SC(\H),\tau)$ is a semi-circular system.
\end{proposition}

Let $L^2(\SC(\H), \tau)$ be the Hilbert space completion of $\SC(\H)$ with the inner product $\ip{A}{B} = \tau(AB^*)$. The following proposition from \cite[Proposition 3]{Biane1997b} relates $L^2(\SC(\H),\tau)$ and $F(\H^\C)$ whose proof uses techniques from \cite{Voiculescu1985}. The Tchebycheff polynomials of type II $(T_k)_{k=1}^\infty$ are defined by the generating function
$$\sum_{n=0}^\infty z^kT_k(x) = \frac{1}{1-xz+z^2}$$
which form a family of complete orthogonal polynomials of the semicircle law.
\begin{proposition}
\label{SemiCircularEvaluation}
Let $(T_k)_{k=0}^\infty$ be the Tchebycheff polynomials of type II, and let $(e_j)_{j=1}^\infty$ be an orthonormal basis of $\H$. Then for any integers $k_1,\cdots, k_n$ and $j_1,\cdots j_n$ such that $j_1\neq j_2\neq\cdots\neq j_n$, we have
$$T_{k_1}(X(e_{j_1}))\cdots T_{k_n}(X(e_{j_n}))=e_{j_1}^{k_1}\tensor\cdots\tensor e_{j_n}^{k_n}.$$
In particular, the map $A\mapsto A\,\Omega$ extends to a unitary isomorphism from $L^2(\SC(\H),\tau)$ to $F(\H^\C)$.
\end{proposition}

\subsection{Free Stochastic calculus}
\label{FreeStochasticSection}
In this section, we will review free (semicircular) Brownian motion, free unitary Brownian motion and free multiplicative Brownian motion as well as free stochastic calculus. For more details of these topics, the reader is directed to \cite{Biane1997, Biane1997b, BianeSpeicher1998, KummererSpeicher1992}; the reader could also read \cite{CollinsKemp2014, Kemp2015} for a simple introduction..
\begin{definition}
A free (semicircular) Brownian motion in a $W^*$-probability space $(\A,\tau)$ is a weakly continuous free stochastic process $(x_t)_{t\geq 0}$ with free and stationary semicircular increments.
\end{definition}
The free Brownian motion can be constructed as a family of operators on a Fock space, cf. \cite{Biane1997b, BianeSpeicher1998}. If we take the real Hilbert space $\H = L^2(\R)$, and define $X_t = X(\1_{[0,t]})$, where $X$ is the sum of creation and annihilation operators defined in Section \ref{SemiCircularSection}, then $X_t$ is a free Brownian motion; in the later sections, we will focus on the concrete realization of the Brownian motions to fit the use of free Segal-Bargmann transform. 

For the von Neumann algebra $\A$, we denote $\A^{\text{op}}$ its opposite algebra, equipped with the trace $\tau^{\text{op}}=\tau$. The reason to consider the opposite algebra is simply because this makes $\A$ and $\A\tensor \A$ have a left $\A\otimes \A^{\text{op}}$ - module structure (here the tensor product is algebraic) in the way that $(a\tensor b)\sharp u=aub$ and $(a\tensor b)\sharp (u\tensor v) = au\tensor vb$. We will also use the $L^2$ completion $L^2(\A\tensor \A^{\text{op}}, \tau\tensor\tau^{\text{op}})$.

A simple biprocess $\Theta_t$ is a piecewise constant map $t\mapsto \Theta_t$ from $\R_+$ into the algebraic tensor product $\A\tensor \A^{\text{op}}$ with $\Theta_t=0$ for all $t$ large enough; it is said to be adapted to $x_t$ if $\Theta_t\in\mathscr{W}_t\tensor \mathscr{W}_t$ where $\mathscr{W}_t = W^*\{x_s: s\leq t\}$, the von Neumann subalgebra generated by $\{x_s\}_{s\leq t}$, for all $t\geq 0$. For $\Theta = A \tensor B \1_{[t_1, t_2]}$ with $A, B\in \mathscr{W}_{t_1}$, we define the free stochastic integral
$$\int \Theta_s\sharp dx_s = A(x_{t_2}-x_{t_1})B$$
and extend the definition linearly to all simple adapted biprocesses.  For an adapted biprocess $\Theta\in L^2(\A\tensor\A^{\text{op}})$, we define the free stochastic integral as an $L^2(\A\otimes \A^{\text{op}},\tau\otimes\tau^{\text{op}})$-limit of step functions of the form $\sum_k \Theta_{t_{k-1}}\sharp(x_{t_k} - x_{t_{k-1}})$ over partitions $\{0=t_0<t_1<\cdots<t_n = t\}$ as the partition width $\sup_j |t_j-t_{j-1}|\to 0$. We define
$$\int_0^t \Theta_s\sharp dx_s = \int \Theta_s\1_{[0,t]}\sharp dx_s.$$

We will frequently write $\int_0^t \theta_s\;dx_s \tilde{\theta}_s$, meaning $\Theta_t = \theta_t\otimes \tilde{\theta}_t$ from the preceding paragraph; the free stochastic integral $\phi_t$  is abbreviated as $d\phi_t = \theta_t\;dx_t \tilde{\theta}_t$. Standard Picard iteration shows that if $h_1, h_2$ are Lipschitz functions, then the left free stochastic differential equation
$$\theta_t = h_1(\theta_t)\;dx_t+h_2(\theta_t)\;dt$$
and the mirrored right integral equation with a given initial condition have a unique adapted solution.

We will also deal with free stochastic integration with respect to two free Brownian motions. Suppose $x_t$ and $y_t$ are freely independent free Brownian motions. We consider the filtration $\mathscr{W}_t = W^*\{x_s, y_s: s\leq t\}$ and the free stochastic integral with respect to $x_t$ and $y_t$ can be defined. Analogous to the free stochastic integration with respect to one free Brownian motion, if $h_1, h_2, h_3$ are Lipschitz functions then the free stochastic differential equation
$$\theta_t = h_1(\theta_t)\;dx_t+h_2(\theta_t)\;dx_t+h_3(\theta_t)\;dt$$
with initial condition $\theta_0$ has a unique adapted solution; any of the integral in the above equation could be changed from a left integral to a right integral.

\begin{remark}
Lipschitz functional calculus only makes sense for self-adjoint or normal processes. If we go beyond the self-adjoint or normal category, we need to restrict the functions $h_k$ concerned in the preceding paragraphs to be polynomials to make sense of everything; however, under this restriction, the Lipschitz requirement holds only for first order polynomials. In fact, K\"{u}mmerer and Speicher \cite{KummererSpeicher1992} considered the functions $h_k$ which are left or right multiplications by a (possibly non-constant) process with extra conditions. Nevertheless, later we will only look at the free stochastic equations for free unitary Brownian motion and free multiplicative Brownian motion; we only need to consider the cases of first order polynomials. 
\end{remark}

The most important computational tool is the free It\^{o} formula whose proof could be found in \cite{BianeSpeicher1998}.
\begin{proposition}
Let $(\A, \tau)$ be a $W^*$-probability space with two freely independent free semicircular Brownian motions $x_t$ and $y_t$. Let $\theta_t, \tilde{\theta}_t$ be adapted processes. Then the following hold:
\begin{align*}
&\tau\left(\int_0^t \theta_s\,dx_s\, \tilde{\theta}_s\right) = \tau\left(\int_0^t \theta_s\,dx_s\,\tilde{\theta}_s\right) = 0;\\
&\int_0^t dx_s \,\theta_s\,dx_s=\int_0^t dy_s\, \theta_s\,dy_s=\int_0^t\tau(\theta_s)\,ds;\\
&\int_0^t dx_s \,\theta_s\,dy_s = \int_0^t dy_s \,\theta_s\,dx_s = 0.
\end{align*}
Moreover, we also have the It\^{o} product rules
$$d(\theta_t\tilde{\theta}_t) = d\theta_t\cdot\tilde{\theta}_t+\theta_t\cdot d\tilde{\theta}_t + d\theta_t\cdot d\tilde{\theta}_t$$
and
$$dx_tdt = dy_tdt = dt^2 = 0.$$
\end{proposition}

\subsection{Free Unitary and Free Multiplicative Brownian Motions}
\label{BM}
\subsubsection{Free Unitary Brownian Motion}
The (left) free unitary Brownian motion was introduced in \cite{Biane1997} as the solution to the free It\^{o} stochastic differential equation
$$du_t = iu_t\,dx_t - \frac{1}{2}u_t \,dt$$
with initial condition $u_1 =1$ where $x_t$ is a free Brownian motion. The adjoint $u_t^*$ satisfies
$$du_t^* = -idx_t\,u_t^* - \frac{1}{2}u_t^* \,dt.$$
The process $u_t$ is a unitary-valued stochastic process whose law is a measure $\nu_t$ on the unit circle, which was introduced in Section \ref{FreeProb}, for each $t>0$. It turns out \cite{Biane1997} that the free unitary Brownian motion has free, stationary multiplicative increments with law $\nu_t$, and is also weakly continuous. The moments of the free unitary Brownian motion, also computed by Biane, are
$$\nu_n(t) := \int \omega^n\,\nu_t(d\omega) = e^{-\frac{nt}{2}}\sum_{k=0}^{n-1}\frac{(-t)^k}{k!}n^{k-1}{n \choose k+1},\;\;\; n\geq 0.$$

\subsubsection{Free Multiplicative Brownian Motion}
\label{fmbm}
Fix $s>\frac{t}{2}>0$. The time of the processes in this section will be denoted by $r$. Let $(\A,\tau)$ be a $W^*$-probability space that contains two freely independent free semicircular Brownian motions $x$ and $y$. Let
$$w_{s,t}(r) = \sqrt{s-\frac{t}{2}}x(r)+\sqrt{\frac{t}{2}}y(r);$$
we will call this a free elliptic $(s,t)$-Brownian motion. A concrete construction on a Fock space of the free elliptic $(s,t)$-Brownian motion will be demonstrated in Section \ref{EllipticSystems} to fit the application of free Segal-Bargmann transform. The free multiplicative Brownian motion $b_{s.t}$ of parameters $s,t$ is the unique solution of the free stochastic differential equation
$$db_{s,t}(r) = ib_{s,t}(r) \,dw_{s,t}(r)-\frac{1}{2}(s-t)b_{s,t}(r)\,dt$$
subject to the initial condition $b_{s,t}(0) = 1$.
\begin{remark}
In \cite{Kemp2015}, Kemp indexed the free multiplicative Brownian motion as $b_{r,s}$ in which the indices $(r,s)$ are a linear change of the indices $(s,t)$ we are using here. Indeed, the $b_{s,t}$ we are considering here is $b_{\scriptscriptstyle s-\frac{t}{2}, \frac{t}{2}}$ in \cite{Kemp2015}. The linear change here is convenient for the discussion on Segal-Bargmann transform.
\end{remark}
When $s=t=1$, $w_{s,t}(r) = \frac{1}{\sqrt{2}}(x(r)+iy(r))$ which is the variance-normalized circular Brownian motion which was studied by Biane \cite{Biane1997, Biane1997b}. In the degenerate case $s=1, t=0$, the process reduces to the free unitary Brownian motion. Kemp \cite[Proposition 1.8]{Kemp2015} computed the moments
\begin{equation}
\label{moments}
\tau[b_{s,t}(r)^n] = \nu_n((s-t)r).
\end{equation}
The free multiplicative $(s,t)$-Brownian motion $b_{s,t}(r)$ is invertible (see, for example, \cite{Kemp2015}), and its inverse satisfies the free stochastic differential equation
$$d(b_{s,t}(r)^{-1}) = -dw_{s,t}(r)\,b_{s,t}(r)^{-1}-\frac{1}{2}(s-t)b_{s,t}(r)^{-1}.$$
The relation between the time and the parameters is that in distribution
$$b_{s,t}(r) = b_{sr, tr}(1).$$
In particular, when $s=t=1$, the process reduces to the one-parameter free multiplicative Brownian motion $b_t = b_{t,t}(1)=b_{1,1}(t)$ in distribution.

\section{The Conditional Expectation Representation}
\label{ConditionalExpectationSection}
In this section, we will relate the two-parameter free unitary Segal-Bargmann transform to a form of conditional expectation. We first review trace polynomials.
\begin{definition}[{\cite{Cebron2013}}]
Let $I$ be an arbitrary index set. 
\begin{enumerate}
\item Denoted by $(\C\{X_i: i\in I\}, \tr, (X_i)_{i\in I})$ the unique (up to an $I$-adapted isomorphism) object which satisfies the universal property:

For all algebras $\mathscr{A}$ with a center-valued trace $\tau$ and for all elements $(A_i)_{i\in I}$ from $\A$, there is a unique homomorphism $\phi$ from $\C\{X_i: i\in I\}$ to $\A$ such that
\begin{enumerate}
\item for all $i\in I$, $\phi(X_i) = A_i$;
\item for all $X\in \C\{X_i: i\in I\}$, we have $\tau(\phi(X)) = \phi(\tr X))$.
\end{enumerate}
$\C\{X_i: i\in I\}$ has a canonical basis given by
$$\{M_0\tr M_1 \cdots \tr M_n: n\in \N, M_0, \ldots M_n \text{ are monomials of } \C\langle X_i: i\in I\rangle\}.$$

\item Consider $\C\{X, X^{-1}\}$, the trace Laurent polynomials with two elements $X$ and $X^{-1}$. Then given any unital complex algebra $\A$, with trace $\tau$, and $A\in \A$ an invertible elements, the map $\phi : \C\{X, X^{-1}\}\to \A$ by
$$\phi(P) = P(A)$$
is a homomorphism from $\C\{X, X^{-1}\}$ to $\A$.

When the unital $\ast$-algebra $\A$ and the invertible element $A\in\A$ is specified, taking the map $\phi$ of an element in $\C\{X, X^{-1}\}$is called evaluation. For convenience, when $\A$ and $A$ are specified, we will write $P(A)$ instead of $\phi(P)$.
\end{enumerate}
\end{definition}
A construction of the space $\C\{X_i: i\in I\}$ is given in the appendix of \cite{Cebron2013}.

When $I$ has only one element, the C\'{e}bron's definition of trace polynomials generated by a single element is isomorphic to the Driver, Hall and Kemp's definition, which is given in the statement of Theorem \ref{DriverHallKemp}; the precise relation is stated in \cite[Lemma 2.3]{CebronKemp2015}. We first summarize some results from C\'{e}bron \cite{Cebron2013}. 
\begin{theorem}[\cite{Cebron2013}]
\label{Cebron}
There exists an operator $\Delta_U$ on the trace polynomials $\C\{X, X^{-1}\}$ such that the one-parameter free unitary Segal-Bargmann transform $\G_t$ satisfies, for all Laurent polynomials $f$,
\begin{equation}
\label{CebronOp}
\G_t f(b_t) = (e^{\frac{t}{2}\Delta_U}f)(b_t)
\end{equation}
where $b_t$ is the free multiplicative Brownian motion and $u_t$ is the free unitary Brownian motion. If in addition, $u_t$ and $g_t$ are free to each other, we have
\begin{equation}
\label{CondExp}
\G_t f(b_t) = \tau(f(b_tu_t) | b_t).
\end{equation}
\end{theorem}
Recall that $L_{\text{hol}}^2(b_t, \tau)$ is the Hilbert space completion of the algebra generated by $b_t$ and $b_t^{-1}$, with norm $\|A\|_2^2 = \tau(A^*A)$. $\G_t f(b_t)$, which is a polynomial in $b_t$, lies in $L_{\text{hol}}^2(b_t, \tau)$. So, Theorem \ref{Cebron} considers the free unitary Segal-Bargmann transform on the operator-side of the Biane-Gross-Malliavin picture.

In general if $z$ and $u_t$ are free to each other, $z\mapsto\tau(f(z u_t) | z)$ gives us a trace polynomial in $z$. Evaluating at $b_t$ in \eqref{CondExp}, all the moments $\tau(b_t^k)$ are evaluated as $1$ (see Section \ref{fmbm} and equation \eqref{moments}). 

It is not hard to combine the work from C\'{e}bron \cite{Cebron2013} and Driver, Hall and Kemp \cite{DriverHallKemp2013} to obtain the two-parameter free unitary Segal-Bargmann transform  $\G_{s,t}$ for polynomials in the form of conditional expectation. Let us quote the result from \cite{DriverHallKemp2013}.
\begin{theorem}[\cite{DriverHallKemp2013}]
\label{DriverHallKemp}
Let $u$ and $\{v_n\}_{n\in\Z}$ be commuting intermediates. Denote $\mathscr{P}^1 = \C[u, u^{-1}]$ the Laurent polynomials in $u$, $\mathscr{P}^0 = \C[v_n: n\in\Z]$ and  $\mathscr{P} = \mathscr{P}^1\tensor \mathscr{P}^0 = \C[u,u^{-1};v_n: n\in\Z]$. 

Let $\pi_s$ be a map on $\mathscr{P}$ which evaluates all the $v_k$ as $\nu_k(s)$ (see definition of $\nu_k$ from equation \eqref{moments}). Then there is an operator $\mathcal{D}$ on $\mathscr{P}$ such that the two-parameter free unitary Segal-Bargmann transform $\G_{s,t}$ is given by
$$\G_{s,t} = \pi_{s-t}\circ \exp\left(\frac{t}{2}\mathcal{D}\right).$$
\end{theorem}

The $u$ from Theorem \ref{DriverHallKemp} is playing the role of a matrix or an operator-valued variable while the $v_k$ is acting as a notation of the kth moment $\tau(u^k)$. The transform $\G_{s,t}$ right now is only defined on Laurent polynomials; it is one of the main purposes of the current paper to extend $\G_{s,t}$ to be a Hilbert space isomorphism.

The $\mathcal{D}$ in Theorem \ref{DriverHallKemp} is the same as the operator $\Delta_U$ in Theorem \ref{Cebron}; see \cite[Lemma 1.19]{DriverHallKemp2013} and \cite[Lemma 4.1]{Cebron2013}.  The formulations of Theorem \ref{Cebron} and Theorem \ref{DriverHallKemp} are just slightly different; C\'{e}bron evaluated all the trace moments by evaluating the trace polynomial at $b_t$, the free multiplicative Brownian motion at time $t$ while Driver, Hall and Kemp defined explicitly the map to evaluate the moments of $b_{s,t}$. The evaluation at $b_t$ in C\'{e}bron's formulation makes his work on the operator side of the Biane-Gross-Malliavin picture. The above observation suggests that if we evaluate $ e^{\frac{t}{2}\Delta_U}f$ at $b_{s,t}$, the trace moments will be evaluated as $\nu_{k}(s-t)$ which gives us the two-parameter free unitary Segal-Bargmann transform.

\begin{proposition}
\label{stCondExpDetailed}
Let $b_{s,t}$ be the free multiplicative $(s,t)$-Brownian motion and $u_t$ be the free unitary Brownian motion. Suppose that the processes $u_t$ and $b_{s,t}$ are freely independent. Then we have, where $b_{s,t} = b_{s,t}(1)$ as an abuse of notation,
\begin{equation*}
\G_{s,t} f(b_{s,t}) = \tau[f(b_{s,t}u_t) | b_{s,t}]
\end{equation*}
for all Laurent polynomials $f$.
\end{proposition}
\begin{proof}
Define $\Phi$ on $\mathscr{P}$, the space defined in Theorem \ref{DriverHallKemp}, with range $\C\{X, X^{-1}\}$ by the algebra isomorphism extension of
$$\Psi(u) = X\;\;\;\Psi(u^{-1})=X^{-1}\;\;\;\Psi(v_n) = \tr(X^n).$$
Note that this map is simply the one mentioned in \cite[Lemma 2.3]{CebronKemp2015}.

The paragraph before this proposition regarding the fact that $\mathcal{D}$ and $\Delta_U$ are the same rigorously means that for any Laurent polynomial $f\in \C[X, X^{-1}]$,
$$e^{\frac{t}{2}\Delta_U}f = \Psi\left(e^{\frac{t}{2}\mathcal{D}}f\right).$$
Thus, since $\pi_{s-t}$ mentioned in Theorem \ref{DriverHallKemp} means to evaluate $v_n$ in $\mathscr{P}$ by the n-th moment $\nu_n(s-t)$ of $b_{s,t}$, we have
$$\G_{s,t} f (b_{s,t}) = \left(\pi_{s-t}\circ \exp\left(\frac{t}{2}\mathcal{D}\right) \right)f(b_{s,t}) = \Psi\left(e^{\frac{t}{2}\mathcal{D}}f\right)(b_{s,t})=e^{\frac{t}{2}\Delta_U}f(b_{s,t})$$
for all Laurent polynomials $f$.

The proof is completed by \cite[Proposition 3.4]{Cebron2013}, which states that for all invertible $B$ which is free from $u_t$,
$$\tau(f(u_t B)|B) = \left(e^{\frac{t}{2}\Delta_U}f\right)(B)$$
for all $f\in\C\{X, X^*, X^{-1}, X^{*-1}\}$.
\end{proof}

\section{The Integral Transform}
\label{IntTrans}

\subsection{Subordination}
\label{Motivation}
Let $u_t$ and $\tilde{u}_t$ be free unitary Brownian motions which are free to each other. Theorem \ref{FreeConv} asserts the existence of a Feller Markov kernel $H = h(\cdot, d\omega)$ on $\U\times\U$ and an analytic function $F$ defined on $\bbD$ such that
$$\tau[f(u_{s-t}\tilde{u}_t) | u_{s-t}] = Hf(u_{s-t})$$
for any bounded Borel function $f$;
$$\int_\U \frac{z\omega}{1-z\omega}\:h(\zeta,d\omega) = \frac{F(z)\zeta}{1-F(z)\zeta}$$
and the analytic function $F$ satisfies
$$\psi_{s-t}(F(z)) = \psi_{\nu_{s-t}\boxtimes\nu_t} =\psi_s(z).$$
where $\psi_{\beta} = \psi_{\nu_\beta}$ as defined in Definition \ref{SigmaTrans}. A simple computation shows that
$$F = \left(\frac{\psi_{s-t}}{1+\psi_{s-t}}\right)^{-1}\circ\left(\frac{\psi_s}{1+\psi_s}\right) = f_{s-t}\circ \chi_s.$$
We define  $\chi_{s,t} = f_{s-t}\circ\chi_s$. It follows that in the $s>t$ case, by Proposition \ref{stCondExpDetailed}, again since $u_{s-t}$ and $b_{s,t}(1)$ have the same holomorphic moments, for all polynomials $f$,
\begin{equation}
\G_{s,t}f(u_{s-t}) = \tau[f(u_{s-t}\tilde{u}_t) | u_{s-t}] = Hf(u_{s-t}) = \int_\U f(\omega)\:h(u_{s-t},d\omega).
\end{equation}
which gives us the integral kernel of the two-parameter free unitary Segal-Bargmann transform and suggests that we look for some substitutes for $s\geq\frac{t}{2}>0$; see Remark \ref{KernelGeneral}.

\subsection{The Two-Parameter Heat Kernel}
\label{TheTwoParaHeatKernel}
Section \ref{Motivation} suggests to look at the subordination function $\chi_{s,t} = f_{s-t}\circ \chi_s$. The map is well defined for all $s,t>0$; nevertheless, we first study some important properties of this map. In order to better understand the statements in the following lemmas, the regions $\Omega_s$ for some $s>0$ (see Proposition \ref{PropOfMaps}) are plotted in the end of the paper. Recall, from Proposition, \ref{PropOfMaps}, $\chi_s$ extends to a homeomorhpism from $\bar{\bbD}$ onto $\bar{\Omega}_s$.

\begin{lemma}
\label{ImageInclusionLem1}
For all $s,t>0$, $f_{s-t}$ maps $\bar{\Omega}_s\cap \bbD$ into $\bbD$ and $\bar{\Omega}_s\cap \U$ into $\U$ ; in particular, $\chi_{s,t}$ maps $\bbD$ into $\bbD$.
\end{lemma}
\begin{proof}
For each $z\in\bar{\Omega}_s$, we have $f_s(z) \in \bar{\bbD}$ and 
$$|f_{s-t}(z)| =\left|e^{\frac{s}{2}\frac{1+z}{1-z}}\right|  \left|e^{-\frac{t}{2}\frac{1+z}{1-z}}\right| = |f_s(z)| e^{-\frac{t}{2}\text{Re}\left(\frac{1+z}{1-z}\right)}.$$
All the conclusions now follow because the M\"{o}bius transform $z\mapsto \frac{1+z}{1-z}$ maps the unit disk onto the right half plane. 
\end{proof}

\begin{lemma}
\label{ImageInclusionLem2}
If $s\geq \frac{t}{2}>0$, $\chi_{s,t}$ is a conformal map from $\bbD$ onto its image and extends to a homeomorphism from $\bar{\bbD}$ to $\overline{\chi_{s,t}(\bbD)}$.
\end{lemma}
\begin{proof}
For $s\geq t$ case, the lemma follows from \cite{Biane1997b} since $\Omega_s$ is increasing in $s$: $\Omega_{s_1}\subseteq \Omega_{s_2}$ if $s_1<s_2$.

Consider the case $\frac{t}{2}\leq s<t$. Since $f_\beta$ is symmetric along the real axis, for all $\beta\in \R$, it suffices to show $\bar{\Omega}_s\cap \C_+\subseteq \bar{\Omega}_{s-t}$. By Proposition \ref{PropOfMaps}, it suffices to prove that $|f_s(re^{i\gamma_{s-t}(r)})|\geq 1$, i.e. $f_s$ maps the boundary of $\bar{\Omega}_{s-t}$ outside the unit disk; recall that $\gamma_{s-t}$ is defined in Proposition \ref{PropOfMaps}. By solving the quadratic equation between $r$ and $\gamma_{s-t}(r)$, we can express $r$ in terms of $\theta$ for $\theta\in [0,\gamma_{s-t}(1)]$: 
$$r = \cos\theta + \frac{(t-s)\sin\theta}{\theta}-\sqrt{\left(\cos\theta + \frac{(t-s)\sin\theta}{\theta}\right)^2-1}.$$
In the rest of the proof, we will denote by $r$ the function of $\theta$ as defined above. Put 
$$x=\text{Re}\:\frac{1+re^{i\theta}}{1-re^{i\theta}}=\frac{2}{t-s}\frac{\theta}{\sin\theta}\sqrt{\left(\cos\theta+\frac{t-s}{2}\frac{\sin\theta}{\theta}\right)^2-1}.$$
We claim that the function $|f_s(re^{i\gamma_{s-t}(r)})| = re^{\frac{s}{2}x}$ of $\theta$ is strictly decreasing on $[0,\gamma_{t-s}(1)]$. Let $\phi = \frac{t-s}{2}x$. Then
$$\frac{dr}{d\theta} + r\frac{d\phi}{d\theta} = \frac{-(t-s)(-1+2\theta^2+\cos 2\theta)\left(\frac{2\theta}{\sin\theta}\right)\left(\frac{t-s}{2}\frac{\sin\theta}{\theta}+\cos\theta-\sqrt{\left(\cos\theta+\frac{t-s}{2}\frac{\sin\theta}{\theta}\right)^2-1}\right)}{4\theta^3\sqrt{\left(2\cos\theta+\frac{(t-s)\sin\theta}{\theta}\right)^2-4}}.$$

Observe that $\frac{t-s}{2}\frac{\sin\theta}{\theta}+\cos\theta-\sqrt{\left(\cos\theta+\frac{t-s}{2}\frac{\sin\theta}{\theta}\right)^2-1}=r>0$ for all $\theta\in [0,\gamma_{s-t}(1)]$. On the other hand, $\frac{d}{d\theta}(-1+2\theta^2+\cos 2\theta) = 4\theta-2\sin2\theta> 0$ on $(0,\gamma_{s-t}(1)]$ and $-1+2\theta^2+\cos 2\theta = 0$ when $\theta=0$. We again have $-1+2\theta^2+\cos 2\theta> 0$ on $(0,\gamma_{s-t}(1)]$ and thus $\frac{dr}{d\theta} + r\frac{d\phi}{d\theta}<0$ on $(0,\gamma_{s-t}(1)]$. Since $\frac{dr}{d\theta}>0$ on $(0, \gamma_{s-t}(1))$ by \ref{PropOfMaps}, $\frac{d\phi}{d\theta}<0$ on $(0, \gamma_{s-t}(1))$. Now our claim follows from
$$\frac{d}{d\theta}re^{\frac{s}{2}x} = \frac{dr}{d\theta}e^{\frac{s}{2}x}+r \frac{s}{2}\frac{dx}{d\theta}e^{\frac{s}{2}x}=e^{\frac{s}{2}x}\left(\frac{dr}{d\theta}+r\frac{s}{t-s}\frac{d\phi}{d\theta}\right)\leq e^{\frac{s}{2}x}\left(\frac{dr}{d\theta}+r\frac{d\phi}{d\theta}\right)<0$$
because if $s\geq\frac{t}{2}$, then $\frac{s}{t-s}\geq\frac{2s}{t}\geq 1$. Notice that $re^{\frac{s}{2}x} = 1$ when $\theta=\gamma_{s-t}(1)$,  we see that $re^{\frac{s}{2}x} > 1$ for all $\theta\in [0,\gamma_{s-t}(1))$. Since $\chi_{s,t}(\bbD)$ is a Jordan domain, by the Carath\'{e}odory's Theorem, it is a conformal map from $\bbD$ onto its image in the disk and extends to a homeomorphism from $\bar{\bbD}$ to $\overline{\chi_{s,t}(\bbD)}$. 
\end{proof}

\begin{lemma}
\label{ImageInclusionLem3}
For $s\geq\frac{t}{2}>0$, $1\not\in \overline{\chi_{s,t}(\bbD)}$. Thus in addition to the proof of Lemma \ref{ImageInclusionLem2} we actually have $\bar{\Omega}_s\cap \C_+\subseteq \Omega_{s-t}$.
\end{lemma}
\begin{proof}
Since $f_{s-t}$ maps $\bar{\Omega}_s\cap \bbD$ into $\bbD$ by Lemma \ref{ImageInclusionLem1}, it suffices to consider $f_{s-t}(\bar{\Omega}_s\cap \U)$. If $s\geq 4$, $\overline{\chi_s(\bbD)}\subseteq \bbD$, the statement follows obviously. If $s< 4$, the case $s>t$ follows from the fact that, since $f_{s-t}$ is a homeomorphism between $\overline{\chi_s(\bbD)}$ and $\bar{\bbD}$, $f_{s-t}(\overline{\chi_s(\bbD)}\cap \U)\subseteq \U\setminus \text{supp}\:\nu_{s-t}$ which does not contain $1$. By the symmetry of $f_{s-t}$ about the real axis, we only need to consider $f(\overline{\chi_s(\bbD)}\cap \U \cap \C_+)$ but by the description of $\Omega_{s-t}$ in Proposition \ref{PropOfMaps} and Remark \ref{IntersectImag}, it suffices to show that $f_{s-t}\left(\frac{i\sqrt{\frac{4}{s}-1}-1}{i\sqrt{\frac{4}{s}-1}+1}\right)\not=1$. 

Let 
$$w=\frac{i\sqrt{\frac{4}{s}-1}-1}{i\sqrt{\frac{4}{s}-1}+1} = 1-\frac{s}{2}+i\frac{s}{2}\sqrt{\frac{4}{s}-1}=e^{i\arccos\left(1-\frac{s}{2}\right)}.$$ 
Observe that $\frac{d}{ds}\arccos\left(1-\frac{s}{2}\right)-\frac{1}{4}\sqrt{s(4-s)} = \frac{2+s}{4\sqrt{s(4-s)}}>0$ when $0<s<4$. We have 
$$\pi> \arccos\left(1-\frac{s}{2}\right)>\arccos\left(1-\frac{s}{2}\right)+\frac{s-t}{2\sqrt{s}}\sqrt{4-s}\geq\arccos\left(1-\frac{s}{2}\right)-\frac{1}{4}\sqrt{s(4-s)}> 0$$
and so
$$we^{\frac{s-t}{2}\frac{1+w}{1-w}} = e^{i\left(\arccos\left(1-\frac{s}{2}\right)+\frac{s-t}{2\sqrt{s}}\sqrt{4-s}\right)}\neq 1.$$

$f_{s-t}$ is a homeomorphism mapping $\overline{\Omega_{s-t}\cap \bbD}$ to $\overline{\C_+\cap\bbD}$. $f_{s-t}$ maps $\U\cap \partial \Omega_{s-t}$ to $\U\cap \R$. But in fact if $w\in \U\cap \partial \Omega_{s-t}$, $f_{s-t}(w)=1$ since $f_{s-t}(-1) = -1$. The last assertion now follows from the above display equation.
\end{proof}

\begin{proposition}
\label{ImageInclusion}
For all $s,t>0$, $f_{s-t}$ maps $\bar{\Omega}_s$ into $\bbD$.  If, in addition, $s\geq\frac{t}{2}$, then $\chi_{s,t}$ is a conformal map from $\bbD$ onto its image, which lies inside the disk, and extends to a homeomorphism from $\bar{\bbD}$ to $\overline{\chi_{s,t}(\bbD)}$; furthermore, $1\not\in \overline{\chi_{s,t}(\bbD)}$.
\end{proposition}
\begin{proof}
Combine Lemmas \ref{ImageInclusionLem1}, \ref{ImageInclusionLem2}, \ref{ImageInclusionLem3}.
\end{proof}

Now the properties of the map $\chi_{s,t}$ is clear. We introduce a notation here.
\begin{definition}
\label{CHI_ST}
Denote $\Omega_{s,t} =\chi_{s,t}(\bbD)\subseteq \bbD$. 
We define $f_{s,t}$ to be the inverse of $\chi_{s,t} = f_{s-t}\circ \chi_s$; its existence is guaranteed by Proposition \ref{ImageInclusion}. It is an analytic function on the domain $\Omega_{s,t}$. 
\end{definition}
Recall from the remark after Theorem \ref{FreeConv} that $\chi_t$ can be analytically continued to $\C_\infty\setminus \text{supp}\:\nu_t$. We now see $\chi_{s,t}$ can be analytically continued to $\C_\infty\setminus \text{supp}\:\nu_s$; the restrictions of $\chi_{s,t}$ to $\bbD$ and $\C\setminus\bbD$ just differ by an inversion.

\begin{theorem}
\label{Nu s,t}
For all $s,t>0$ and all $\zeta\in \U$, there exists a unique probability measure $k_{s,t}(\zeta, d\omega)$ on $\U$ characterized by the moment-generating function
$$\int_\U \frac{z\omega}{1-z\omega}\; k_{s,t}(\zeta, d\omega) = \frac{\chi_{s,t}(z)\zeta}{1-\chi_{s,t}(z)\zeta}.$$
If $s=t$, $k_{s,t}(\zeta, d\omega) = k_s(\zeta, d\omega)$.
\end{theorem}
Even though only $s\geq\frac{t}{2}>0$ is of our interest, $k_{s,t}$ nevertheless exists for all $s,t>0$ since $\chi_{s,t}$ is well-defined on $\bbD$.

\begin{proof}
This is a simple consequence of Herglotz's Representation Theorem to, for each $\zeta\in \U$, the holomorphic function
$$z\mapsto \frac{1+\chi_{s,t}(z)\zeta}{1-\chi_{s,t}(z)\zeta}$$
defined on $\bbD$.
\end{proof}

\begin{remark}
\label{KernelGeneral}
The construction of the kernel $k_{s,t}$ from Theorem \ref{Nu s,t} is what we have been looking for since the introduction of the kernel $h$ in Section \ref{Motivation}. The kernels $k_{s,t}$ and $h$ are characterized by the same moment generating function, or equivalently, the same complex Poisson integral in terms of $\chi_{s,t}$. Of course, due to the nice behavior of the map $\chi_{s,t}$ when $s\geq\frac{t}{2}>0$, we will focus our attention to this case and construct the two-parameter free unitary Segal-Bargmann transform.
\end{remark}

In \cite{Biane1997b}, Biane showed that when $s=t$, the measure $k_{t}(\zeta, d\omega)$ is absolutely continuous with respect to the measure $\nu_t$, for each $t>0$, $\zeta\in \Sigma_t\cap \U$. The result holds also for $\nu_{s,t}$ for $s\geq\frac{t}{2}>0$, $\zeta\in\Sigma_{s,t}\cap\U$ where $\Sigma_{s,t}=\C_\infty\setminus(\overline{\chi_{s,t}(\C_\infty\setminus\text{supp}\:\nu_s)})=\C_\infty\setminus(f_{s-t}(\C_\infty\setminus\bar{\Sigma}_s))$ for $s\geq \frac{t}{2}$. Some pictures of the regions $\Sigma_{s,t}$ are included in the end of the paper. Note that Biane also plotted some regions $\Sigma_t$ in \cite{Biane1997b} which are the $s=t$ cases.

\begin{proposition}
\label{NewDensity}
For $s\geq \frac{t}{2}>0$, and $\zeta\in\Sigma_{s,t}\cap\U$, the measure $k_{s,t}(\zeta, d\omega)$ is absolutely continuous with respect to the measure $\nu_{s}$, with density
$$\frac{|1-\chi_{s}(\omega)|^2}{(\zeta-\chi_{s,t}(\omega))(\zeta^{-1}-\bar{\chi}_{s,t}(\omega))}\frac{1-|\chi_{s,t}(\omega)|^2}{1-|\chi_{s}(\omega)|^2}.$$
\end{proposition}
\begin{proof}
The idea of the proof is exactly the same as the proof of \cite[proposition 12]{Biane1997b} except replacing $\chi_t$ by $\chi_{s,t}$.

For all $z\in\bbD$, $\zeta\in\U$,
$$\frac{\zeta+\chi_{s,t}(z)}{\zeta-\chi_{s,t}(z)}=\begin{cases}
\frac{\displaystyle\frac{\zeta-1}{\zeta+1}+\kappa_{s,t}(z)}{\displaystyle 1+\kappa_{s,t}(z)\frac{\zeta-1}{\zeta+1}}\;\;&\text{if }\zeta\neq -1\\
\frac{\displaystyle 1}{\displaystyle\kappa_{s,t}(z)} &\text{if }\zeta= -1.
\end{cases}$$
$\frac{\zeta-1}{\zeta+1}$ is purely imaginary and the M\"{o}bius transform
$$w\mapsto\frac{\displaystyle\frac{\zeta-1}{\zeta+1}+w}{\displaystyle 1+w\frac{\zeta-1}{\zeta+1}}$$
maps the right half plane into itself. If $\zeta\in\Sigma_{s,t}\cap\U$, the functions 
$$z\mapsto \frac{\zeta+\chi_{s,t}(z)}{\zeta-\chi_{s,t}(z)}$$
are bounded in $\bar{\bbD}$ and is a homeomorphism of $\bar{\bbD}$ with a bounded region in the right half plane. Hence, by Herglotz's Representation Theorem, $k_{s,t}(\zeta,d\omega)$ is absolutely continuous with respect to $d\omega$ with density
$$\text{Re}\:\left(\frac{\zeta+\chi_{s,t}(z)}{\zeta-\chi_{s,t}(z)}\right)=\frac{1-|\chi_{s,t}(\omega)|^2}{(\zeta-\chi_{s,t}(\omega))(\zeta^{-1}-\bar{\chi}_{s,t}(\omega))}.$$
By Lemma \ref{ImageInclusionLem3}, $1\not\in \chi_{s,t}(\U)$.  

So, in particular, we have $\nu_s(d\omega) = k_{s,s}(1, d\omega)$ given by (which is also a result from \cite{Biane1997b})
$$\nu_s(d\omega) = \frac{1-|\chi_{s}(\omega)|^2}{|1-\chi_{s}(\omega)|^2} \;d\omega.$$

Because $|\chi_{s,t}(\omega)|=1$ if and only if $|\chi_s(\omega)|=1$ by Lemma \ref{ImageInclusionLem1}, $\nu_s$ and $k_{s,t}(\zeta, d\omega)$ have the same support and
$$k_{s,t}(\zeta, d\omega) = \frac{|1-\chi_{s}(\omega)|^2}{(\zeta-\chi_{s,t}(\omega))(\zeta^{-1}-\bar{\chi}_{s,t}(\omega))}\frac{1-|\chi_{s,t}(\omega)|^2}{1-|\chi_{s}(\omega)|^2}\;\nu_s(d\omega)$$

\end{proof}

\subsection{The Two-Parameter Free Unitary Segal-Bargmann Transform}
\label{IntegralTransform}
In \cite{Biane1997b}, Biane defined the integral transform $\G_t$ on $L^2(\nu_t)$ as
$$\G_t f(\zeta) = \int_\U f(\omega)\frac{|1-\chi_t(\omega)|^2}{(\zeta-\chi_t(\omega))(\zeta-\bar{\chi}_t(\omega))}\;\nu_t(d\omega)$$
which converges for all $\zeta\in\Sigma_t$. $\G_t f$ is analytic on $\Sigma_t$. 

The issue raised is that if we would like to a priori integrate $f(\omega)\in L^2(\nu_s)$ against the measure $k_{s,t}(\zeta,d\omega)$, would the integral converge? The answer is affirmative; in fact, they have the same $L^2$ functions:

\begin{proposition}
\label{AbsCont}
For $s\geq \frac{t}{2}$ and $\zeta\in \Sigma_{s,t}\cap\U$, the measures $\nu_s$ and $k_{s,t}(\zeta, d\omega)$ are absolutely continuous with respect to each other with bounded densities; in particular, they have the same support.
\end{proposition}
\begin{proof}
By proposition \ref{NewDensity}, we have
$$k_{s,t}(\zeta, \omega) = \frac{|1-\chi_{s}(\omega)|^2}{(\zeta-\chi_{s,t}(\omega))(\zeta^{-1}-\bar{\chi}_{s,t}(\omega))}\frac{1-|\chi_{s,t}(\omega)|^2}{1-|\chi_{s}(\omega)|^2}\nu_s(d\omega).$$

It remains to prove that the density is bounded above and bounded away from $0$. By Lemma \ref{ImageInclusionLem3} and the fact that $\Sigma_{s,t}$ is open, 
$$\frac{|1-\chi_{s}(\omega)|^2}{(\zeta-\chi_{s,t}(\omega))(\zeta^{-1}-\bar{\chi}_{s,t}(\omega))}$$
is bounded above and bounded away from $0$. The function
$$\omega\mapsto \frac{1-|\chi_{s,t}(\omega)|^2}{1-|\chi_{s}(\omega)|^2}$$
is well defined and bounded on $\text{supp}\:\nu_s$, by an application of the L'  H\^{o}pital's Rule;
$$\lim_{\omega\to \alpha} \frac{1-|\chi_{s,t}(\omega)|}{1-|\chi_{s}(\omega)|} = \lim_{\omega\to \alpha} \frac{1-\left|\omega e^{-\frac{t}{2}\frac{1+\chi_s(\omega)}{1-\chi_s(\omega)}}\right|}{1-\left|\omega e^{-\frac{s}{2}\frac{1+\chi_s(\omega)}{1-\chi_s(\omega)}}\right|}=\lim_{\omega\to \alpha} \frac{1-e^{-\frac{t}{2}\text{Re}\:\frac{1+\chi_s(\omega)}{1-\chi_s(\omega)}}}{1-e^{-\frac{s}{2}\text{Re}\:\frac{1+\chi_s(\omega)}{1-\chi_s(\omega)}}}=\frac{t}{s}$$
where $\alpha = e^{\pm i\left(\frac{1}{2}\sqrt{t(4-t)}+\arccos\left(1-\frac{t}{2}\right)\right)}$. 
\end{proof}
The above proposition shows that it makes sense to make the following definition.
\begin{definition}
Let $s\geq \frac{t}{2}>0$. For each $f\in L^2(\nu_s)$, we define
$$\tilde{\G}_{s,t}f(\zeta)=\int_\U f(\omega) \frac{|1-\chi_{s}(\omega)|^2}{(\zeta-\chi_{s,t}(\omega))(\zeta^{-1}-\bar{\chi}_{s,t}(\omega))}\frac{1-|\chi_{s,t}(\omega)|^2}{1-|\chi_{s}(\omega)|^2}\nu_s(d\omega)$$
for all $\zeta\in\Sigma_{s,t}$.
\end{definition}
\begin{remark}
Using standard arguments, for example, applying Morera's Theorem, $\tilde{\G}_{s,t}f$ defines an analytic function on $\Sigma_{s,t}$.
\end{remark}
We shall note that $\tilde{\G}_{s,t}f = \int_\U f(\omega)\;k_{s,t}(\;\cdot\;, d\omega)$ on $\U\cap \Sigma_{s,t}$, from the proof of Proposition \ref{NewDensity}. We will show $\tilde{\G}_{s,t} = \G_{s,t}$.\\

Having defined the integral transform, we are concerned with its range. In the case of Lie group $G$, the Segal-Bargmann-Hall transform is an isomorphism between an $L^2$ space and a holomorphic $L^2$ space. In the case of $s=t$, Biane proved, in \cite{Biane1997b}, that $\G_t$ is an isomorphism between $L^2(\nu_t)$ and a reproducing kernel Hilbert space when $t\neq 4$. Before we discuss the range of the integral transform $\tilde{\G}_{s,t}$, we first find a Cauchy integral representation of it.

\begin{lemma}
\label{DiffChi s,t}
For $s\geq \frac{t}{2}>0$, we have
$$\left(1+\frac{t\:\chi_{s-t}(\chi_{s,t}(z))}{(\chi_{s-t}(\chi_{s,t}(z))-1)^2+(s-t)\chi_{s-t}(\chi_{s,t}(z))}\right)\chi_{s,t}'(z)z=\chi_{s,t}(z)$$
for $z\in \bar{\bbD}$, where if $z\in \U$, the derivative means to differentiate along the curve $\U$.
\end{lemma}
\begin{remark}
Even though in Proposition \ref{PropOfMaps} when $t<0$ the map $\chi_t$ is defined on $\Omega_t\subseteq \C_+$  with range in $\C_+$, the map $f_t$, which is symmetric about the real axis, is one-to-one on $\Omega_t\cup(-\infty, z_-(t))\cup(z_+(t),\infty)\cup\text{conj}\:\Omega_t$ where $\text{conj}\:\Omega_t = \{\bar{z}: z\in\Omega_t\}$. Thus, $\chi_{s-t}(\chi_{s,t}(z))$ makes sense even for $\chi_{s,t}(z)$ in the lower half plane and $s<t$.
\end{remark}
\begin{proof}
Differentiating $\chi_{s,t}(z) = f_{s-t}\circ \chi_s(z) = ze^{-\frac{t}{2}\frac{1+\chi_s(z)}{1-\chi_s(z)}}$ gives $\chi_{s,t}'=(f_{s-t}'\circ \chi_s) \cdot \chi_s'$ and 
$$\chi_{s,t}'(z) =e^{-\frac{t}{2}\frac{1+\chi_s(z)}{1-\chi_s(z)}}+z\left(\frac{-t\:e^{-\frac{t}{2}\frac{1+\chi_s(z)}{1-\chi_s(z)}}}{(\chi_s(z)-1)^2)}\right)\chi_s'(z).$$
Rearranging the terms, we have 
$$\chi_{s,t}'(z)e^{\frac{t}{2}\frac{1+\chi_s(z)}{1-\chi_s(z)}}+\frac{tz\chi_s'(z)}{(\chi_s(z)-1)^2} =1.$$
Since $e^{\frac{t}{2}\frac{1+\chi_s(z)}{1-\chi_s(z)}} = \frac{z}{\chi_{s,t}(z)}$ and $\chi_s' = \frac{\chi_{s,t}'}{f_{s-t}'\circ\chi_s}$, the above equation is the same as
$$z\chi_{s,t}'(z)\left(\frac{1}{\chi_{s,t}(z)} + \frac{t}{(\chi_s(z)-1)^2}\frac{1}{f_{s-t}'\circ\chi_s(z)}\right)=1$$
which is, after computing $f_{s-t}'(z)=e^{\frac{s-t}{2}\frac{1+z}{1-z}}\left(\frac{1+(s-t-2)z+z^2}{(z-1)^2}\right)$,
$$z\chi_{s,t}'(z)\left(1+\frac{t\chi_{s,t}(z)}{e^{\frac{s-t}{2}\frac{1+\chi_s(z)}{1-\chi_s(z)}}(1+(s-t-2)\chi_s(z)+\chi_s(z)^2)}\right)=\chi_{s,t}(z).$$
Now, the result follows from $\chi_{s,t}(z) = \chi_s(z)e^{\frac{s-t}{2}\frac{1+\chi_s(z)}{1-\chi_s(z)}}$, $1+(s-t-2)\chi_s(z)+\chi_s(z)^2=(\chi_s(z)-1)^2+(s-t)\chi_s(z)$ and $\chi_s(z) = \chi_{s-t}\circ \chi_{s,t}(z)$.
\end{proof}

Now we are ready to give the Cauchy integral representation of $\G_{s,t}$.

\begin{proposition}
\label{Cauchy}
Suppose $g\in L^2(\nu_s)$, then for any $\zeta\in\Sigma_{s,t}$, 
$$\tilde{\G}_{s,t}g(\zeta) = \frac{1}{2\pi i}\int_{\partial \Sigma_{s,t}} g(f_{s,t}(z))\left(1+\frac{t\:\chi_{s-t}(z)}{(\chi_{s-t}(z)-1)^2+(s-t)\chi_{s-t}(z)}\right)\,\frac{dz}{z-\zeta}.$$
\end{proposition}
\begin{proof}
Since $\tilde{\G}_{s,t}g$ is analytic, it suffices to prove the result for all $\zeta\in\Sigma_{s,t}\cap\U$. Let $h_s = \frac{1}{2}\sqrt{(4-s)s}+\arccos(1-s/2)$ if $s\leq 4$, $h_s = \pi$ if $t>4$ so that the arc $\{e^{i\theta}: -h_s\leq \theta \leq h_s\}$ is the support of $\nu_{s,t}$. Then we have, from the fact that $\bar{\chi}_{s,t}(e^{i\theta}) = \chi_{s,t}(e^{-i\theta})$, for all $\zeta\in\U\cap\Sigma_{s,t}$,
\begin{equation*}
\begin{split}
\tilde{\G}_{s,t}g(\zeta) = &\frac{1}{2\pi}\int_{-h_s}^{h_s} g(e^{i\theta})\frac{1}{2}\left(\frac{\zeta+\chi_{s,t}(e^{i\theta})}{\zeta-\chi_{s,t}(e^{i\theta})}+\frac{\zeta^{-1}+\bar{\chi}_{s,t}(e^{i\theta})}{\zeta^{-1}-\bar{\chi}_{s,t}(e^{i\theta})}\right)\;d\theta\\
=&\frac{1}{2\pi}\int_{-h_s}^{h_s} g(e^{i\theta})\frac{1}{2}\left(\frac{\zeta+\chi_{s,t}(e^{i\theta})}{\zeta-\chi_{s,t}(e^{i\theta})}-\frac{\zeta+\chi_{s,t}(e^{-i\theta})^{-1}}{\zeta-\chi_{s,t}(e^{-i\theta})^{-1}}\right)\;d\theta\\
=&\frac{1}{2\pi}\int_{-h_s}^{h_s} g(e^{i\theta})\frac{\chi_{s,t}(e^{i\theta})}{\zeta-\chi_{s,t}(e^{i\theta})}\;d\theta-\frac{1}{2\pi}\int_{-h_s}^{h_s}g(e^{i\theta})\frac{\chi_{s,t}(e^{-i\theta})^{-1}}{\zeta-\chi_{s,t}(e^{-i\theta})^{-1}}\;d\theta.\\
\end{split}
\end{equation*}
For the first term, we apply Lemma \ref{DiffChi s,t} to get
\begin{equation*}
\begin{split}
&\frac{1}{2\pi}\int_{-h_s}^{h_s} g(e^{i\theta})\frac{\chi_{s,t}(e^{i\theta})}{\zeta-\chi_{s,t}(e^{i\theta})}\;d\theta\\
=&\frac{1}{2\pi i}\int_{-h_s}^{h_s} g(e^{i\theta})\left(1+\frac{t\:\chi_{s-t}(\chi_{s,t}(z))}{(\chi_{s-t}(\chi_{s,t}(e^{i\theta}))-1)^2+(s-t)\chi_{s-t}(\chi_{s,t}(e^{i\theta}))}\right)\;\frac{\chi_{s,t}'(e^{i\theta})ie^{i\theta}\;d\theta}{\zeta-\chi_{s,t}(e^{i\theta})}\\
=&\frac{1}{2\pi i}\int_{\partial \Sigma_{s,t}\cap \bar{\bbD}} g(f_{s,t}(z))\left(1+\frac{t\:\chi_{s-t}(z)}{(\chi_{s-t}(z)-1)^2+(s-t)\chi_{s-t}(z)}\right)\,\frac{dz}{z-\zeta};\\
\end{split}
\end{equation*}
we have used the fact that the parametrization $\theta\mapsto \chi_{s,t}(e^{i\theta})$ ($\theta\in[-h_s, h_s]$) for $\partial \Sigma_{s,t}\cap \bar{\bbD}$ is negative-oriented. For $\partial \Sigma_{s,t}\cap (\C_\infty\setminus\bbD)$, we use the parametrization $\theta\mapsto \chi_{s,t}(e^{-i\theta})^{-1}$ ($\theta\in [-h_s, h_s]$) which satisfies a similar relation in Lemma \ref{DiffChi s,t}; the parametrization for this part is positive-oriented so 
\begin{equation*}
\begin{split}
&-\frac{1}{2\pi}\int_{-h_s}^{h_s}g(e^{i\theta})\frac{\chi_{s,t}(e^{-i\theta})^{-1}}{\zeta-\chi_{s,t}(e^{-i\theta})^{-1}}\;d\theta\\
=&\frac{1}{2\pi i}\int_{\partial \Sigma_{s,t}\cap (\C_\infty\setminus\bbD)} g(f_{s,t}(z))\left(1+\frac{t\:\chi_{s-t}(z)}{(\chi_{s-t}(z)-1)^2+(s-t)\chi_{s-t}(z)}\right)\,\frac{dz}{z-\zeta}.
\end{split}
\end{equation*}
\end{proof}

\subsection{The Range of the Transform}
In this subsection, we will show that the transform $\tilde{\G}_{s,t}$ has range into the Hardy space $H^2(\Sigma_{s,t})$ (see Definition \ref{Hardy}). We will also find the range of $\tilde{\G}_{s,t}$ for which $\tilde{\G}_{s,t}$ is a unitary isomorphism. We will prove that $\tilde{\G}_{s,t}$ and $\G_{s,t}$ coincide on polynomials as well.
\begin{lemma}
\label{Z}
For all $s\geq \frac{t}{2}>0$, define an operator $Z_{s,t}$ on $L^2(\nu_{s})$ by
$$Z_{s,t}g(z) = g(f_{s,t}(z))\left(1+\frac{t\:\chi_{s-t}(z)}{(\chi_{s-t}(z)-1)^2+(s-t)\chi_{s-t}(z)}\right)$$
for all $g\in L^2(\nu_{s})$. $Z_{s,t}$ has range in $L^2(\sigma_{s,t})$ where $\sigma_{s,t}$ is the arc-length measure on $\Sigma_{s,t}$. If $s\neq 4$, there are constants $c, C$ such that
$$c\|g\|_{L^2(\nu_{s})}\leq \|Z_{s.t}g\|_{L^2(\sigma_{s,t})}\leq C\|g\|_{L^2(\nu_{s})}$$
for all $g\in L^2(\nu_{s})$. If $s=4$, there is a constant $C$ such that
$$\|Z_{s.t}g\|_{L^2(\sigma_{s,t})}\leq C\|g\|_{L^2(\nu_{s})}$$
for all $g\in L^2(\nu_{s})$.
\end{lemma}
\begin{proof}
By Proposition \ref{AbsCont}, it suffices to prove the lemma with $L^2(k_{s,t}(1, d\omega))$ instead of $L^2(\nu_s)$. In this proof, we will denote $k_{s,t}(1, d\omega)$ by $\nu_{s,t}$, to shorten the notation. Recall from the proof of Proposition \ref{NewDensity} that $\nu_{s,t}(d\omega) = \frac{1-|\chi_{s,t}(\omega)|^2}{|1-\chi_{s,t}(\omega)|^2} d\omega$. The image of the measure $\nu_{s,t}$ under the map $\chi_{s,t}$ is absolutely continuous with respect to $\sigma_{s,t}$ with density
$$\left|\frac{1-|z|^2}{|1-z|^2}\left(1+\frac{t\chi_{s-t}(z)}{(\chi_{s-t}(z)-1)^2+(s-t)\chi_{s-t}(z)}\right) \right|$$
on $\partial \Sigma_{s,t}\cap \bar{\bbD}$; similarly, the image of $\nu_{s,t}$ under the map $1/\bar{\chi}_{s,t}$ is absolutely continuous with respect to $\sigma_{s,t}$ with density
$$\left|\frac{1-|z|^2}{|1-z|^2}\left(1+\frac{t\chi_{s-t}(z)}{(\chi_{s-t}(z)-1)^2+(s-t)\chi_{s-t}(z)}\right)\right|.$$
Let $g\in L^2(\nu_{s,t})$. Then
\begin{equation*}
\begin{split}
\|g\|_{L^2(\nu_{s,t})}=&\int_{\partial \Sigma_{s,t}}|g(f_{s.t}(z))|^2\left|\frac{1-|z|^2}{|1-z|^2}\left(1+\frac{t\chi_{s-t}(z)}{(\chi_{s-t}(z)-1)^2+(s-t)\chi_{s-t}(z)}\right)\right|\sigma_{s,t}(dz)\\
=&\int_{\partial \Sigma_{s,t}}|Z_{s,t}g(z)|^2\left|\frac{1-|z|^2}{|1-z|^2}\left(1+\frac{t\chi_{s-t}(z)}{(\chi_{s-t}(z)-1)^2+(s-t)\chi_{s-t}(z)}\right)^{-1}\right|\sigma_{s,t}(dz)\\
=&\int_{\partial \Sigma_{s,t}}|Z_{s,t}g(z)|^2\left|\frac{1-|z|^2}{|1-z|^2}\frac{(\chi_{s-t}(z)-1)^2+(s-t)\chi_{s-t}(z)}{(\chi_{s-t}(z)-1)^2+s\chi_{s-t}(z)}\right|\sigma_{s,t}(dz).
\end{split}
\end{equation*}
If $s>4$, since $\chi_{s,t}(\U)\subseteq \bbD$ and $\chi_{s,t}'=f_{s-t}'\circ\chi_s\cdot \chi_s'$ never vanishes ($f_{s-t}$ only vanishes at $z_\pm(s-t)$ mentioned in Proposition \ref{PropOfMaps} and $f_s'$ is bounded on $\Sigma_s$), there are constants $0<c, C<\infty$ such that 
$$0<c<\left|\frac{1-|z|^2}{|1-z|^2}\frac{(\chi_{s-t}(z)-1)^2+(s-t)\chi_{s-t}(z)}{(\chi_{s-t}(z)-1)^2+s\chi_{s-t}(z)}\right|<C<\infty.$$
If $s<4$, note that by Lemma \ref{DiffChi s,t} and $f_{s,t}'=f_s'\circ\chi_{s-t}\cdot \chi_{s-t}'$ is bounded above on the compact set $\partial \Sigma_{s,t}$ (by Lemma \ref{ImageInclusionLem3}, $\chi_{s-t}'$ makes sense on $\partial \Sigma_{s,t}$), $1+\frac{t\chi_{s-t}(z)}{(\chi_{s-t}(z)-1)^2+(s-t)\chi_{s-t}(z)}$ is bounded above. We have that
$$\frac{(\chi_{s-t}(z)-1)^2+(s-t)\chi_{s-t}(z)}{|1-z|^2}$$
is bounded below away from $0$ and above. On the other hand,
$$\frac{1-|z|^2}{(\chi_{s-t}(z)-1)^2+s\chi_{s-t}(z)}=\frac{(1-|z|)(1+|z|)}{(\chi_{s-t}(z)-b_s)(\chi_{s-t}(z)-\bar{b}_s)}$$
where $b_s$ and $\bar{b}_s$ are the intersections of $\partial \Sigma_s$ with $\U$ since by Proposition \ref{PropOfMaps} $\sqrt{\frac{(x+1)^2-(x-1)^2e^{sx}}{e^{sx}-1}}\to i\sqrt{\frac{4}{s}-1}$. Since the derivatives of the function, for $0<s<4$, $s\mapsto \sqrt{\frac{(x+1)^2-(x-1)^2e^{sx}}{e^{sx}-1}}$ goes to $0$ as $x\to 0$, by Proposition \ref{PropOfMaps}, the curves $\partial \Sigma_s$ intersects $\U$ orthogonally; it follows that $(1-|z|)/(\chi_{s-t}(z)-b_s)$ and $(1-|z|)/(\chi_{s-t}(z)-\bar{b}_s)$ are bounded above and below from $0$ for $z\in \partial \Sigma_{s,t}$. Therefore,
$$0<c<\left|\frac{1-|z|^2}{|1-z|^2}\frac{(\chi_{s-t}(z)-1)^2+(s-t)\chi_{s-t}(z)}{(\chi_{s-t}(z)-1)^2+s\chi_{s-t}(z)}\right|<C<\infty.$$
If $s=4$, $\partial \Sigma_{s,t}$ fails to intersect $\U$ orthogonally, so we only have
$$0<c<\left|\frac{1-|z|^2}{|1-z|^2}\frac{(\chi_{s-t}(z)-1)^2+(s-t)\chi_{s-t}(z)}{(\chi_{s-t}(z)-1)^2+s\chi_{s-t}(z)}\right|$$
but not the bounded above side.
\end{proof}

We first quote the definition of the Hardy space of a general domain from \cite{Rudin1955}.
\begin{definition}
\label{Hardy}
For any domain $G\subseteq \C$, we define the Hardy space $H^2(G)$ as the set containing all the analytic functions $f$ on $G$ for which there exists a harmonic function $u$ on $G$ such that
$$|f|^2\leq u$$
on the domain $G$.
\end{definition}

Proposition \ref{Cauchy} and Lemma \ref{Z} have an immediate consequence .

\begin{proposition}
For $s\geq \frac{t}{2}>0$, $\tilde{\G}_{s,t}$ is a bounded map mapping $L^2(\nu_s)$ into the Hardy space $H^2(\Sigma_{s,t})$.
\end{proposition}
\begin{proof}
By Proposition \ref{Cauchy} and Lemma \ref{Z}, 
$$\tilde{\G}_{s,t}g(\zeta) = \frac{1}{2\pi i}\int_{\partial \Sigma_{s,t}} \frac{Z_{s,t}g(z)}{z-\zeta}\;dz,$$
and $Z_{s,t}g$ is in the $L^2$ space of the arc-length measure of $\partial \Sigma_{s,t}$.
 Since $\Sigma_{s,t}$ is a domain with piecewise analytic boundary, the proposition follows from a direct application of \cite[Theorem 3.2]{Rudin1955}. 
\end{proof}

Before we completely describe the range of $\tilde{\G}_{s,t}$, we prove one more important property.
\begin{lemma}
$\tilde{\G}_{s,t}:L^2(\nu_{s})\to H^2(\Sigma_{s,t})$ is injective, for all $s\neq 4$.
\end{lemma}
\begin{proof}
We shall separate the cases $s<4$ and $s>4$. For $s>4$, the compact set $\chi_{s,t}(\bar{\bbD})$ is contained in $\bbD$; therefore, the power series expansion
$$\text{Re}\:\left(\frac{\zeta+\chi_{s,t}(\omega)}{\zeta-\chi_{s,t}(\omega)}\right)= 1+\sum_{n=1}^\infty\left(\zeta^{-n}\chi_{s,t}(\omega)^n+\zeta^n\bar{\chi}_{s,t}(\omega)^n\right)$$
converges uniformly for $\zeta,\omega\in\U$. The measure $\nu_s$ is equivalent to the Haar measure $d\omega$ by Proposition \ref{NewDensity}; thus $L^2(d\omega)=L^2(\nu_s)$. Let $g\in L^2(d\omega)$ such that $\tilde{\G}_{s,t}g= 0$ on $\U$. Now, 
\begin{equation}
\label{CoeffZero}
\begin{split}
\tilde{\G}_{s,t}g(\zeta) &= \int_\U g(\omega) \text{Re}\:\left(\frac{\zeta+\chi_{s,t}(\omega)}{\zeta-\chi_{s,t}(\omega)}\right)\:d\omega\\
&=\int_\U g \:d\omega+\sum_{n=1}^\infty \zeta^{-n} \int_\U g(\omega)\chi_{s,t}(\omega)^n\;d\omega+\zeta^n\int_\U g(\omega)\bar{\chi}_{s,t}(\omega)^nd\omega
\end{split}
\end{equation}
is the Fourier expansion of $\tilde{\G}_{s,t}g$. That $\tilde{\G}_{s,t}g= 0$ on $\U$ implies that all the Fourier coefficients 
$$\int_\U g(\omega)\chi_{s,t}(\omega)^n\;d\omega=\int_\U g(\omega)\bar{\chi}_{s,t}(\omega)^n d\omega=0.$$
Since the continuous function $\chi_{s,t}$ is one-to-one, it separates points on $\U$. By the Stone-Weierstrass Theorem, polynomials in $\chi_{s,t}$ and $\bar{\chi}_{s,t}$ are dense in the continuous functions on $\U$; hence we can approximate $L^2(d\omega)$ functions by polynomials in $\chi_{s,t}$ and $\bar{\chi}_{s,t}$. It follows from \eqref{CoeffZero} that  $$\int_\U g \varphi \; d\omega = 0$$
for all $\varphi \in L^2(d\omega)$, which implies
$g= 0$.\\

Now let $s<4$ and let $\tilde{\G}_{s,t}g=0$. Then the Cauchy Transform of $Z_{s,t}g$ is $0$ as in Lemma \ref{Cauchy}. By C\'{a}lderon's Theorem, this implies that $Z_{s,t}g$ is the boundary function of the function
$$\zeta\mapsto \frac{1}{2\pi i}\int_{\partial \Sigma_{s,t}} g(f_{s,t}(z))\left(1+\frac{t\:\chi_{s-t}(z)}{(\chi_{s-t}(z)-1)^2+(s-t)\chi_{s-t}(z)}\right)\,\frac{dz}{z-\zeta}$$
for $\zeta\in\C\setminus\bar{\Sigma}_{s,t}$ which vanishes at infinity and is in the Hardy space of the region $\C\setminus \bar{\Sigma}_{s,t}$. Now, we will define four conformal maps $h_1$, $h_2$, $h_3$, $h_4$ so that the composition of the four functions maps $\C_\infty\setminus\bar{\Sigma}_{s,t}$ onto $\C_\infty\setminus\bar{\bbD}$ as follows.
\begin{enumerate}
\item
The map $h_1(z) = f_{s,t}(z)$ is an injective conformal map from $\C_\infty\setminus \bar{\Sigma}_{s,t}$ onto $\C_\infty\setminus\text{supp}\:\nu_{s,t}$. It should be highlighted that $h_1(\bar{z}) = \overline{h_1(z)}$; this fact follows from the same reflection property for $f_{s-t}$ and $\chi_s$.
\item
Let $h_2(z) = 2\frac{a_s+1}{a_s-1}\frac{z-1}{z+1}$, which is an injective conformal map from $\C_\infty\setminus\text{supp}\:\nu_{s,t}$ onto $\C_\infty\setminus [-2,2]$ where $a_s = h_1(b_s)$, $b_{s,t}\in\C_+\cap \U\cap\partial\Sigma_{s,t}$ and $a_s$ is the endpoint of the connected arc of $\text{supp}\:\nu_{s,t}$. We observe that $h_2(\bar{z})=-\overline{h_2(z)}$, from the fact that $\bar{a}_s$ is the other endpoint of $\text{supp}\;\nu_s$, an arc on the unit circle.
\item
Define $$h_3(z)=\frac{1}{2}(z+\sqrt{z^2-4})=\int_{-2}^2 \frac{1}{2\pi}\frac{\sqrt{4-x^2}}{x-z}\;dx$$
where the square root is chosen so that this defines a conformal map from $\C_\infty\setminus[-2,2]$ onto $\C_\infty\setminus\bar{\bbD}$. 
Note that $h_3(\bar{z}) = \overline{h_3(z)}$, $h_3(-z) = -h_3(z)$ and $h_3(i\R)\subseteq i\R$; the second and third properties follow easily from a simple substitution $x\mapsto -x$ in the Cauchy integral. (The function $h_3$ is an inverse of the Jakowski transform $J(z) = z+1/z$, from which these properties also easily follow.)
\item
We finally denote $h_4(z)= i \frac{1-z \bar{c}_s}{z-c_s}$, where $c_s = h_3\circ h_2(\infty)$. Because $h_2(\infty)$ is not in the interval $[-2,2]$, $c_s$ is a point in $\C_\infty\setminus \bar{\bbD}$. The map $h_4$ is a M\"{o}bius transform mapping $\C_\infty\setminus\bbD$ onto itself (and similarly from $\bbD$ onto itself) since $h_4(0) = -i/c_s\in\bbD$.
\end{enumerate}
We now consider the composition $\Phi=h_4\circ h_3\circ h_2\circ h_1$ which defines an injective conformal map from $\C_\infty\setminus\bar{\Sigma}_{s,t}$ onto $\C_\infty\setminus\bar{\bbD}$. Here is a list of important properties of $\Phi$:
\begin{enumerate}
\item[(i)] $\Phi(\infty) = \infty$; this follows from $h_1(\infty) = \infty$, $h_2(\infty) = 2\frac{a_s+1}{a_s-1}$, $h_3(h_2(\infty)) = c_s$ and $h_3(c_s) = \infty$.
\item[(ii)] $\Phi(\bar{z}) = \overline{\Phi(z)}$ for all $z$. This can be seen by direction computations that 
$h_2\circ h_1(\bar{z}) =h_2(\overline{h_1(z)})=-\overline{h_2\circ h_1(z)}$, $h_3(-\bar{w}) = -\overline{h_3(w)}$ and $h_4(-\bar{w}) = \overline{h_4(w)}$.

\item[(iii)] $\Phi'$ and $1/\Phi'$ are bounded on $\C_\infty\setminus\Sigma_{s,t}$. We first list out the zeros and the poles of all the $h_i'$. The map $h_1'$ has simple zeros at $b_{s,t}$ and $\bar{b}_{s,t}$ (which follows directly from the fact that the derivative of its inverse $\chi_{s,t}' = (f_{s-t}'\circ \chi_s)\cdot \chi_s'$ only has poles at the endpoints of $\text{supp}\;\nu_{s,t}$) as well as a double pole at $\infty$; $h_2'$ has a pole of order $2$ at $-1$ and a zero of order $2$ at $\infty$; $h_3'$ has simple poles at $J(1)=2$ and $J(-1)=-2$ as well as a zero of order $2$ at $\infty$; $h_4'$ has poles and zeros, both of order $2$, at $c_s$ and $\infty$ respectively. Then we conclude that the zeros and the poles are all cancelled out when we apply the chain rule to the derivative of $\Phi$, the composition of the four analytic functions. It follows that $\Phi'$ is bounded above and bounded away from $0$.

\item[(iv)] $\Phi$ has an analytic continuation to a neighborhood of $\C_\infty\setminus\Sigma_{s,t}$. When the analytic function $f_{s,t}$ is restricted to $(\C_\infty\setminus \Sigma_{s,t})\cap \bbD$, it has an analytic continuation to a neighborhood of the closure of $(\C_\infty\setminus \Sigma_{s,t})\cap \bbD$; this is because $f_{s,t} = f_s\circ \chi_{s-t}$ and both $f_s$ and $\chi_{s-t}$ are actually defined on a larger domain. Similarly, on $(\C_\infty\setminus \Sigma_{s,t})\cap (\C\setminus\bbD)$, $f_{s,t}$ has an analytic continuation to a neighborhood of the closure of $(\C_\infty\setminus \Sigma_{s,t})\cap \ (\C\setminus\bbD)$ since $f_{s,t}$ preserves inversion. It follows that $f_{s,t}$ has an analytic continuation to a neighborhood of $\C_\infty\setminus \Sigma_{s,t}$. All the other $h_i$ are nice enough to compose with the analytic continued $f_{s,t}$. 
\end{enumerate}
Whence $\phi\mapsto \phi\circ\Phi^{-1}$ is a bounded map, with bounded inverse, from $H^2(\C_\infty\setminus \bar{\Sigma}_{s,t})$ onto $H^2(\C_\infty\setminus\bar{\bbD})$ so that $Z_{s,t}f\circ\Phi^{-1}$ is the boundary function of some function in $H^2(\C_\infty\setminus\bar{\bbD})$. The map $z\mapsto Z_{s,t}f\circ\Phi^{-1}(1/z)$ is the boundary function of some function in $H^2(\bbD)$ which vanishes at $0$. 
Let 
$$\Psi(z)=h_4^{-1}(1/\Phi(z)).$$ 
By construction, $\Phi(b_{s,t}) = h_4\circ h_3\circ h_2(a_s) = h_4\circ h_3(2) = h_4(1)$ and similarly $\Phi(\bar{b}_{s,t}) = h_4(-1)$. By property (ii) of $\Phi$ mentioned above, 
$$\Psi(b_{s,t}) = h_4^{-1}(1/\Phi(b_{s,t}))=h_4^{-1}(\overline{\Phi(b_{s,t})})=h_4^{-1}(h_4(-1))=-1$$
and similarly $\Psi(\bar{b}_{s,t}) = 1$.

We claim that $$\Psi(1/\bar{z})  = \bar{\Psi}(z)$$ for all $z\in\C_\infty\setminus\bar{\Sigma}_{s,t}$. Fix a $z$ and set $\zeta=h_3\circ h_2\circ h_1(z)$. Since $h_1(1/z) = 1/h_1(z)$, $h_2(1/z) = -h_2(z)$ and $h_3(-z) = -h_3(z)$, we have $\Phi(1/z) = h_4(-\zeta)$ and $\Phi(1/\bar{z})=h_4(\bar{\zeta})$. Thus $$\Psi(1/\bar{z}) = h_4^{-1}\circ \frac{1}{h_4}(\bar{\zeta})$$
and
$$\Psi(z) = h_4^{-1}\circ \frac{1}{h_4}(\zeta).$$
Now, the claim follows easily from explicit computations that $h_4^{-1}\circ (1/h_4)(\zeta) = -1/\zeta$, using the fact that $c_s$ is purely imaginary.

Since we have
$$1+\frac{t\:\chi_{s-t}(\Psi^{-1}(z))}{(\chi_{s-t}(\Psi^{-1}(z))-1)^2+(s-t)\chi_{s-t}(\Psi^{-1}(z))}=\frac{(\chi_{s-t}\circ\Psi^{-1}(z)-b_s)(\chi_{s-t}\circ\Psi^{-1}(z)-\bar{b}_s)}{(\chi_{s-t}\circ\Psi^{-1}(z)-1)^2+(s-t)\chi_{s-t}\circ\Psi^{-1}(z)},$$
the function
$$\left(1+\frac{t\:\chi_{s-t}(\Psi^{-1}(z))}{(\chi_{s-t}(\Psi^{-1}(z))-1)^2+(s-t)\chi_{s-t}(\Psi^{-1}(z))}\right)^{-1}(1-z^2)$$
is holomorphic and bounded on $\bbD$, so that the function
$$g_{s,t}(z) = Z_{s,t}g\circ\Phi^{-1}(z)\left(1+\frac{t\:\chi_{s-t}(\Psi^{-1}(z))}{(\chi_{s-t}(\Psi^{-1}(z))-1)^2+(s-t)\chi_{s-t}(\Psi^{-1}(z))}\right)^{-1}(1-z^2)$$
is the boundary function of some function in $H^2(\bbD)$ which vanishes at $\Phi(\infty)$. By the definition of this function and the properties of $\Psi$ and $\chi_{s-t}$ which gives us $\chi_{s-t}\circ \Phi(1/\bar{z}) = \overline{\chi_{s-t}\circ\Phi(z)}$, we see that $g_{s,t}(e^{-i\theta}) = -e^{-2i\theta}g_{s,t}(e^{i\theta})$ on $\U$. Since $g_{s,t}(e^{i\theta}) = \sum_{n=0}^\infty a_n e^{in\theta}$, we have $a_n=0$ for all $n\neq 0,2$ and $a_0=-a_2$. Since $g_{s,t}$ is the boundary function of an $H^2(\bbD)$ function which vanishes at some point in $\bbD$, we conclude $g_{s,t}=0$ hence $g=0$.
\end{proof}

\begin{remark}
This proof followed the proof of \cite[Lemma 17]{Biane1997b}, but that proof had typos in the definition of $h_4$ and used a third function $h_5$ which should have been $h_5 = h_4^{-1}$.
\end{remark}

We finally are able to state the following theorem.
\begin{theorem}
\label{IntMainThm}
The transform $\tilde{\G}_{s,t}$ is an unitary isomorphism between the Hilbert spaces $L^2(\nu_s)$ and the reproducing kernel Hilbert space $\mathscr{A}_{s,t}$ of analytic functions on $\Sigma_{s,t}$ generated by the positive-definite sesqui-analytic kernel
$$K_{s,t}(z,\zeta)=\int_\U\frac{|1-\chi_{s}(\omega)|^2}{(z-\chi_{s,t}(\omega))(z^{-1}-\bar{\chi}_{s,t}(\omega))}\frac{|1-\chi_{s}(\omega)|^2}{(\bar{\zeta}-\bar{\chi}_{s,t}(\omega))(\bar{\zeta}^{-1}-\chi_{s,t}(\omega))}\left(\frac{1-|\chi_{s,t}(\omega)|^2}{1-|\chi_s(\omega)|^2}\right)^2\;\nu_{s}(d\omega)$$
\end{theorem}
\begin{proof}
For each $z\in\Sigma_{s,t}$, let 
$$h_z(\omega) = \frac{1-|\chi_{s,t}(\omega)|^2}{1-|\chi_s(\omega)|^2}\frac{|1-\chi_{s,t}(\omega)|^2}{(z-\chi_{s,t}(\omega))(z^{-1}-\bar{\chi}_{s,t}(\omega))}$$
be defined on $\U$.
We will drop the subscripts $s$ and $t$ for $K$ in this proof and denote $K_\zeta(z) = K(z,\zeta)$ as an analytic function on $\Sigma_{s,t}$. Observe that for any $f\in L^2(\U, \nu_{s})$, 
$$\G f(\zeta) = \ip{f}{\bar{h}_\zeta}_{L^2(\U, \nu_{s})}.$$

Define $\A_{s,t} \equiv \G(L^2(\U, \nu_{s}))$ equipped with an inner product
$$\ip{F}{G}_{\A_{s,t}} \equiv \ip{\G^{-1}F}{\G^{-1}G}_{L^2(\U, \nu_{s})}$$
which is well-defined since $\G$ is injective. $\G$ is an isometry from $L^2(\U, \nu_{s})$, which is a Hilbert space, onto $\A_{s,t}$ and $\G$ has range in $H^2(\Sigma_{s,t})$; thus $(\A_{s,t}, \ip{\cdot}{\cdot}_{\A_{s,t}})$ does define a Hilbert space of analytic functions. By the construction here, $\G$ is a unitary isomorphism between $L^2(\U, \nu_{s})$ and $\A_{s,t}$.

It is easy to see that $K_\zeta(z) = \G \bar{h}_\zeta(z)$ and for any $F\in \A_{s,t}$,
$$\ip{F}{K_\zeta}_{\A_{s,t}} = \ip{\G^{-1}F}{\bar{h}_\zeta}_{L^2(\U,\nu_{s})}=\G(\G^{-1}F)(\zeta) = F(\zeta).$$
This shows that $K$ is a reproducing kernel for $\A_{s,t}$.
\end{proof}

\begin{lemma}
\label{GenFunCoincide}
Let $P_{s,t}^{(n)}$ be the polynomials determined by the generating function
$$\sum_{n=1}^\infty z^n P_{s,t}^{(n)}(u)=\frac{f_{s,t}(z)u}{1-f_{s.t}(z)u}$$
ane denote $P_{s,t}^{(n)*}(u) = P_{s,t}^{(n)}(\bar{u})$. Then we have
$$\tilde{\G}_{s,t}P_{s,t}^{(n)}(\zeta) = \zeta^n$$
and
$$\tilde{\G}_{s,t}P_{s,t}^{(n)*}(\zeta) = \zeta^{-n}.$$
\end{lemma}
\begin{proof}
As in the proof of Theorem \ref{Nu s,t}, we can compute, without difficulty, that the moment generating function of $k_{s,t}(\zeta, d\omega)$ is
$$\int_\U \frac{z\omega}{1-z\omega}\;k_{s,t}(\zeta, d\omega) = \frac{\chi_{s,t}(z)\zeta}{1-\chi_{s,t}(z)\zeta}.$$
so that
$$\int_\U \frac{f_{s,t}(z)\omega}{1-f_{s,t}(z)\omega}\;k_{s,t}(\zeta, d\omega)=\frac{z\zeta}{1-z\zeta}.$$
$\tilde{\G}_{s,t}P_{s,t}^{(n)}(\zeta) = \zeta^n$ follows by expanding both sides in power series with respect to $z$ and then analytic continuation of $\tilde{\G}_{s,t}P_{s,t}^{(n)}(\zeta)$ to $\Sigma_{s,t}$. 

$\tilde{\G}_{s,t}P_{s,t}^{(n)*}(\zeta) = \zeta^{-n}$ follows without difficulties. In view of the proof of Proposition \ref{NewDensity}, $k_{s,t}(\zeta, d\omega)$ is absolutely continuous with respect to the Haar measure $d\omega$ on $\U$ with density
$$\frac{1-|\chi_{s,t}(\omega)|^2}{(\zeta-\chi_{s,t}(\omega))(\zeta^{-1}-\bar{\chi}_{s,t}(\omega))}.$$
Replacing $\omega$ by $\bar{\omega}$ from the density, we have, by $\chi_{s,t}(\bar{\omega}) = \bar{\chi}_{s,t}(\omega)$,
$$\frac{1-|\chi_{s,t}(\bar{\omega})|^2}{(\zeta-\chi_{s,t}(\bar{\omega}))(\zeta^{-1}-\bar{\chi}_{s,t}(\bar{\omega}))} = \frac{1-|\chi_{s,t}(\omega)|^2}{(\zeta-\bar{\chi}_{s,t}(\omega))(\zeta^{-1}-\chi_{s,t}(\omega))} .$$
Therefore,
$$\int_\U \frac{f_{s,t}(z)\bar{\omega}}{1-f_{s,t}(z)\bar{\omega}}\;k_{s,t}(\zeta, d\omega)=\int_\U \frac{f_{s,t}(z)\omega}{1-f_{s,t}(z)\omega}\;k_{s,t}(\zeta^{-1}, d\omega)=\frac{z\zeta^{-1}}{1-z\zeta^{-1}}.$$
The result follows a priori by expanding both sides with respect to $z$ and the observation that 
$$\sum_{n=1}^\infty z^n P_{s,t}^{(n)*}(u)=\frac{f_{s,t}(z)\bar{u}}{1-f_{s.t}(z)\bar{u}}.$$
\end{proof}

\begin{theorem}
\label{IntCoincides}
The transform $\tilde{\G}_{s,t}$ coincides to $\G_{s,t}$ on polynomials, the large $N$-limit of the Segal-Bargmann transform on $\U(N)$; i.e. $\tilde{\G}_{s,t}$ extends $\G_{s,t}$ to a unitary isomorphism between the two Hilbert spaces.
\end{theorem}
\begin{proof}
By Lemma \ref{GenFunCoincide}, the generating functions of the transform $\tilde{\G}_{s,t}$ and the large $N$-limit of the Segal Bargmann transform on $\U(N)$ coincide.
\end{proof}
We will from now on write $\G_{s,t}$ instead of $\tilde{\G}_{s,t}$. That the range of $\G_{s,t}$ contains all the polynomials has the simple result:
\begin{corollary}
The reproducing kernel Hilbert space $\mathscr{A}_{s,t}$ is a dense subspace of $H^2(\Sigma_{s,t})$.
\end{corollary}
\begin{remark}
We shall note that even in the case $s=t$, we do not know whether the Hilbert spaces $\mathscr{A}_{s,t}$ and $H^2(\Sigma_{s,t})$ coincide; however, as mentioned in \cite{Biane1997b}, they are equipped with different inner products and there is no measure $m$ on $\C$ such that
$$\int_\C F(z)\bar{G}(z)\;m(dz) = \langle F, G\rangle_{\mathscr{A}_{s,t}}$$
for all entire functions on $\C$, at least in the case $s=t$, and presumably in general.
\end{remark}

\subsection{Limiting Behavior as $s\to\infty$}



In this section we start by showing the following theorem on the asymptotic behavior of the boundary of $\Sigma_{s,t}$ when $s\to \infty$.
\begin{theorem}
\label{ConvCurves}
If we fix $t>0$, $|\chi_{s,t}|\to e^{-\frac{t}{2}}$ as $s\to\infty$ uniformly on $\U$. In the case $s=t$, $e^{\frac{t}{2}}|\chi_{t,t}|\to 1$.
\end{theorem}

Since $\chi_{s,t}$ is the inner boundary curve of the region $\Sigma_{s,t}$, the above theorem actually says that with $t>0$ fixed, the region $\Sigma_{s,t}$ converges to an annulus with inner and outer radii $e^{-\frac{t}{2}}$ and $e^{\frac{t}{2}}$ respectively as $s\to \infty$ (see Corollary \ref{ConvRegion}); in the case $s=t$, $\Sigma_{t,t}$ approaches  (but not converges) to an annulus of inner and outer radii $e^{-\frac{t}{2}}$ and $e^{\frac{t}{2}}$ respectively.\\

\begin{proof}
Recall that $\chi_{s,t}$ denotes the inner curves (inside the $\bbD$) of the region $\Sigma_{s,t}$. We can compute the modulus of $\chi_{s,t}$ easily:
$$|\chi_{s,t}(z)| =\exp \left(-\frac{t}{2}\text{Re}\;\frac{1+\chi_s(z)}{1-\chi_s(z)}\right).$$
So we attempt to estimate $\text{Re}\;\frac{1+\chi_s(z)}{1-\chi_s(z)}$ for large $t$. As mentioned in the first part of Proposition \ref{PropOfMaps}, $x=\text{Re}\;\frac{1+\chi_t(z)}{1-\chi_t(z)}$ satisfies
$$\phi_t(x):=\left|\frac{x-1}{x+1}\right|e^{\frac{t}{2}x}<1.$$

For $t>4$, there exists $0<\alpha_-(t)<1<\alpha_+(t)$ such that $\phi_t<1$ on $[\alpha_-(t),\alpha_+(t)]$ (see \cite{Biane1997b}). We will approximate $\alpha_\pm(t)$. Observe that $\phi_t$ is differentiable on $(0,1)$, $\phi_t'(0)>0$, $\phi_t(0) = 1$ and $\phi_t(1) =0$; there is a point $\sqrt{1-4/t}$ such that $\phi_t'(\sqrt{1-4/t}) = 0$. Thus, $\alpha_-(t) > \sqrt{1-4/t}$ and $\phi_t'(x) < 0$ for all $\sqrt{1-4/t}<x<1$. Therefore, $x^2-1 > -4/t$ for all $\sqrt{1-4/t}<x<1$ and so
$$\phi_t''(x) = \frac{e^{tx/2}}{4(1+x)^3}(-16+8t(1+x)+t^2(x^2-1)(x+1))>0.$$
Now, by the Fundamental Theorem of Calculus,
$$-1 = \phi_t(1) - \phi_t(\alpha_-(t)) = \int_{\alpha_-(t)}^1 \phi_t' \leq \phi_t'(1)(1-\alpha_-(t)) = -\frac{1}{2}e^{\frac{t}{2}}(1-\alpha_-(t))$$
and hence $1-\alpha_-(t) \leq 2 e^{-\frac{t}{2}}$.\\

We now estimate $\alpha_+(t)$. We first compute for $t>4$, $x>1$, in this case, it is even easier to get
$$\phi_t''(x) = \frac{e^{tx/2}}{4(1+x)^3}(-16+8t(1+x)+t^2(x-1)(x+1)^2)>0.$$
So $\phi_t$ is convex and
$$1=\phi_t(\alpha_+(t))-\phi_t(1) = \int_1^{\alpha_+(t)}\phi_t'\geq \frac{t}{2}e^{\frac{t}{2}}(\alpha_+(t)-1)$$
where the quantity $\frac{1}{2}e^{\frac{t}{2}}$ comes from the fact that $\left.\frac{d}{dx}\frac{x-1}{x+1}e^{\frac{t}{2}x}\right|_{x=1}=\frac{1}{2}e^{\frac{t}{2}}$. It follows that
$$\alpha_+(t)\leq 1+2e^{-\frac{t}{2}}.$$
Whence $|\alpha_\pm(t)-1|\leq 2e^{-\frac{t}{2}}$ when  $t$ large enough. 

We are ready for the estimates to
$$|\chi_{s,t}(z)| =\exp \left(-\frac{t}{2}\text{Re}\;\frac{1+\chi_s(z)}{1-\chi_s(z)}\right).$$
When $s\neq t$, $|\alpha_\pm(s)-1|\leq 2e^{-\frac{s}{2}}$ implies $\text{Re}\;\frac{1+\chi_s(z)}{1-\chi_s(z)}\to 1$ uniformly. So if we fix $t$ and let $s\to\infty$, 
$$|\chi_{s,t}(z)|\to e^{-\frac{t}{2}}$$
uniformly. For the case $s=t$, 
$$\frac{e^{-\frac{t}{2}}}{|\chi_{t,t}(z)|} =\exp \left(\frac{t}{2}\left(\text{Re}\;\frac{1+\chi_t(z)}{1-\chi_t(z)}-1\right)\right).$$
But the same estimate for $\alpha_\pm(t)$ shows that $\frac{t}{2}\left(\text{Re}\;\frac{1+\chi_t(z)}{1-\chi_t(z)}-1\right)\to 0$ uniformly as $t\to\infty$.
\end{proof}

Now it comes to the formulation of the observation we made in the beginning of this section.

\begin{corollary}
\label{ConvRegion}
When $t>0$ is fixed, $\bar{\Sigma}_{s,t} \to A_t$ in Hausdorff distance as $s\to \infty$,
where $A_t = \{z\in\C: e^{-\frac{t}{2}}\leq |z|\leq e^{\frac{t}{2}}\}$ is the annulus with inner and outer radii $e^{-\frac{t}{2}}$ and $e^{\frac{t}{2}}$ respectively.
\end{corollary}
\begin{proof}
Fix $t>0$. We first notice that $\bar{\Sigma}_{s,t}\subseteq A_t$ for all $s > \frac{t}{2}$. Recall that for all $s>4$,  $\bar{\Sigma}_{s,t} = \C\setminus(\chi_{s,t}(\bbD)\cup 1/\chi_{s,t}(\bbD))$. The region $\chi_{s,t}(\bbD)$ is a simply connected, with boundary $\chi_{s,t}(\U)$; it contains $B(0, e^{-\frac{t}{2}})$.

Let $\varepsilon > 0$. Denote $m_s = \max_{z\in \chi_{s,t}(\U)} |z|$ which exists since the curve $\chi_{s,t}(\U)$ is compact. By Theorem \ref{ConvCurves}, there is an $s_0>4$ such that for all $s>s_0$, $m_s - e^{-\frac{t}{2}}<\varepsilon$. If there exists $|w| = e^{-\frac{t}{2}}$ such that $w$ is not in the $\varepsilon$-neighborhood of $\chi_{s,t}(\U)$ for some $s>s_0$, then for this $s$, $\chi_{s,t}(\bbD)$ cannot contain $B(0, e^{-\frac{t}{2}})$ since $\chi_{s,t}(\bbD)$ is a region with boundary curve $\chi_{s,t}(\U)$. A similar statement holds for $1/\chi_{s,t}(\U)$. It follows that the $\varepsilon$-neighborhood of $\bar{\Sigma}_{s,t}$ contains $A_t$ for all $s > s_0$.
\end{proof}

\section{The Biane-Gross-Malliavin Theorem}
\label{BianeGrossMalliavin}

\subsection{Elliptic Systems and Free Segal-Bargmann Transform}
\label{EllipticSystems}
The (classical) Segal-Bargmann transform is a unitary isomorphism between the Hilbert spaces $L^2(\R^n,\rho_s)$, where $\rho_s$ is the Gaussian density with variance $s$, and $L_{\text{hol}}^2(\C^n, \rho_{s,t})$, where $\rho_{s,t}$ is a two-parameter heat kernel due to Driver and Hall \cite{DriverHall1999}. The $W^*$-probability space $(\SC(\H),\tau)$ defined in Section \ref{SemiCircularSection} is a candidate of a free analogue of $L^2(\R^n,\rho_s)$. In this section, we will introduce and construct an $(s,t)$-elliptic system on the Fock space which generalizes the circular system introduced in \cite{Biane1997b} and plays the role of $L^2(\C^n, \rho_{s.t})$ in the free context; we will then define the free $(s,t)$-Segal-Bargmann transform between the semi-circular and $(s,t)$-elliptic systems.

\begin{definition}
\label{EllipticDef}
Suppose that $s\geq \frac{t}{2}>0$. A linear map $c: \H\to \A$ is called an $(s,t)$-elliptic system if
\begin{enumerate}
\item there exist semi-circular systems $x$ and $y$ such that $c = \sqrt{s-\frac{t}{2}}x+i\sqrt{\frac{t}{2}}y$;
\item the subsets $\{x(h): h\in\H\}$ and $\{y(h): h\in\H\}$ are free in $(\A,\tau)$.
\end{enumerate}
\end{definition}

We shall construct the $(s,t)$-elliptic systems on a free Fock space, parallel to the construction of circular system in \cite{Biane1997b, Kemp2005}. Let $\H$ be a real Hilbert space and $\H^\C$ its complexification. We consider the sum $X(h) = a_h+a_h^*$ of annihilation and creation operators defined on the Fock space $F((\H\oplus\H)^\C)$ (See Section \ref{SemiCircularSection}). For each $h\in\H$, we define
$$Z_{s,t}(h) = \sqrt{s-\frac{t}{2}}X(h,0)+i\sqrt{\frac{t}{2}}X(0,h).$$
Then $Z_{s,t}$ is a $(s,t)$-elliptic system on the $W^*$-probability space $\EX^{s,t}(\H)=W^*\{Z(h): h\in\H\}$ equipped with the canonical vacuum state $\tau$. Let $\EX_{\text{hol}}^{s,t}(\H)$ be the Banach algebra generated by $\{Z(h): h\in\H\}$ and $L^2(\EX_{\text{hol}}^{s,t}(\H)),\tau)$ its Hilbert space completion under the inner product $\ip{A}{B} = \tau(AB^*)$. We now compute the action of $Z(h)$ on $(\H\oplus\H)^{\tensor n}$:
\begin{equation*}
\begin{split}
&Z_{s,t}(h)(h_1,g_1)\otimes\cdots\otimes (h_n, g_n) \\
= &\sqrt{\scriptstyle s-\frac{t}{2}}\big[(h,0)\otimes (h_1,g_1)\otimes\cdots\otimes (h_n, g_n)+\ip{h_1}{h}(h_2,g_2)\otimes\cdots\otimes (h_n, g_n)\big]\\
&+i\sqrt{\scriptstyle\frac{t}{2}}\big[(0,h)\otimes (h_1,g_1)\otimes\cdots\otimes (h_n, g_n)+\ip{g_1}{h}(h_2,g_2)\otimes\cdots\otimes (h_n, g_n)\big]\\
=&\left(\sqrt{\scriptstyle s-\frac{t}{2}}h,i\sqrt{\scriptstyle\frac{t}{2}}h\right)\otimes (h_1,g_1)\otimes\cdots\otimes (h_n, g_n)+\ip{\sqrt{\scriptstyle s-\frac{t}{2}}h_1+i\sqrt{\scriptstyle \frac{t}{2}}g_1}{h} (h_2,g_2)\otimes\cdots\otimes (h_n, g_n).
\end{split}
\end{equation*}
If all $(h_k, g_k)$ are of the form $\left(\sqrt{\scriptstyle s-\frac{t}{2}}h_k, i\sqrt{\scriptstyle\frac{t}{2}}h_k\right)$, we have
\begin{equation}
\begin{split}
\label{Z(h)Action}
&Z_{s,t}(h)\left(\sqrt{\scriptstyle s-\frac{t}{2}}h_1, i\sqrt{\scriptstyle\frac{t}{2}}h_1\right)\otimes\cdots\otimes\left(\sqrt{\scriptstyle s-\frac{t}{2}}h_n, i\sqrt{\scriptstyle\frac{t}{2}}h_n\right)\\
=&(\sqrt{\scriptstyle s-\frac{t}{2}}h,i\sqrt{\scriptstyle\frac{t}{2}}h)\otimes \left(\sqrt{\scriptstyle s-\frac{t}{2}}h_1, i\sqrt{\scriptstyle\frac{t}{2}}h_1\right)\otimes\cdots\otimes\left(\sqrt{\scriptstyle s-\frac{t}{2}}h_n, i\sqrt{\scriptstyle\frac{t}{2}}h_n\right)\\
&+(s-t)\ip{h_1}{h} (\sqrt{\scriptstyle s-\frac{t}{2}}h_2,i\sqrt{\scriptstyle\frac{t}{2}}h_2)\otimes\cdots\otimes (\sqrt{\scriptstyle s-\frac{t}{2}}h_n,i\sqrt{\scriptstyle\frac{t}{2}}h_n).
\end{split}
\end{equation}
In particular, when all $h_k = h$ with $\|h\| = 1$, we have
$$Z_{s,t}(h)\left(\sqrt{\scriptstyle s-\frac{t}{2}}h, i\sqrt{\scriptstyle\frac{t}{2}}h\right)^{\tensor n} = \left(\sqrt{\scriptstyle s-\frac{t}{2}}h, i\sqrt{\scriptstyle\frac{t}{2}}h\right)^{\tensor (n+1)}+(s-t)\left(\sqrt{\scriptstyle s-\frac{t}{2}}h, i\sqrt{\scriptstyle\frac{t}{2}}h\right)^{\tensor (n-1)}$$
where the latter term is $0$ when $n-1<0$. We shall also note that when $h$ and $k$ are orthogonal to each other, from the computation of equation \eqref{Z(h)Action}, we have
\begin{equation*}
\label{Z(h)onOrthogonal}
Z_{s,t}(h)\left(\sqrt{\scriptstyle s-\frac{t}{2}}k, i\sqrt{\scriptstyle\frac{t}{2}}k\right)^{\tensor n}=(\sqrt{\scriptstyle s-\frac{t}{2}}h,i\sqrt{\scriptstyle\frac{t}{2}}h)\otimes \left(\sqrt{\scriptstyle s-\frac{t}{2}}k, i\sqrt{\scriptstyle\frac{t}{2}}k\right)^{\tensor n}.
\end{equation*}
We define $\delta_{s,t}(h) = \frac{1}{\sqrt{s}}\left(\sqrt{\scriptstyle s-\frac{t}{2}}h, i\sqrt{\scriptstyle\frac{t}{2}}h\right)$ for $h\in\H^\C$ and extend it to $F(\H^\C)$ by $\delta_{s,t}(h_1\tensor \cdots\tensor h_n) = \delta_{s,t}(h_1)\tensor\cdots\tensor \delta_{s,t}(h_n)$. The extension $\delta_{s,t}$ on $F(\H^\C)$ is an isometry. The above computations prove the following proposition:
\begin{proposition}
\label{EllipticEvaluation}
\begin{enumerate}
\item The sequence $(Q_{s,t}^{(n)})_{n=1}^\infty$ of Tchebycheff type II polynomials with parameter $s-t$ satisfying the recurrence relation 
$$ Q_{s,t}^{(n+1)}(x)=xQ_{s,t}^{(n)}(x)-(s-t)Q_{s,t}^{(n-1)}(x)$$
with $Q_{s,t}^{(0)}(x) = 1$, $Q_{s,t}^{(1)}(x) = x$ has the property that for any $h\in\H$,
$$Q_{s,t}^{(n)}(Z_{s,t}(h))\Omega = \left(\sqrt{\scriptstyle s-\frac{t}{2}}h_1, i\sqrt{\scriptstyle\frac{t}{2}}h_1\right)^{\tensor n}.$$
\item Let $(e_j)_{j=1}^\infty$ be an orthonormal basis of $\H$. For any integers $k_1,\cdots, k_n$ and $j_1,\cdots j_n$ such that $j_1\neq j_2\neq\cdots\neq j_n$, we have
$$Q_{s,t}^{(k_1)}(Z_{s,t}(e_{j_1}))\cdots Q_{s,t}^{(k_n)}(Z_{s,t}(e_{j_n}))\Omega = \delta_{s,t}(e_{j_1})^{k_1}\tensor\cdots\tensor \delta_{s,t}(e_{j_n})^{k_n}.$$
\item The map $A\mapsto A\,\Omega$ extends to a unitary isomorphism from $L^2(\EX_{\text{hol}}^{s,t}(\H),\tau)$ to $\delta_{s,t}(F(\H))$.
\end{enumerate}
\end{proposition}
\begin{remark}
When $s=t$, the polynomials $Q_{s,t}^{(n)}$ are monomials; on the other hand, $Q_{1,0}^{(n)}$ are the Tchebycheff type II polynomials with parameter $1$. Thus, in general $Q_{s,t}^{(n)}$ interpolate, or extrapolate due to the sign of $s-t$, between the two kinds of polynomials.
\end{remark}
We shall now define the free $(s,t)$-Segal-Bargmann transform which gives (up to the map $\delta_{s,t}$) unitary equivalence between $L^2(\SC(\H),\tau)$ and $L^2(\EX_{\text{hol}}^{s,t}(\H),\tau)$. 
\begin{definition}
\label{FreeSegalBargmann}
The free $(s,t)$-Segal-Bargmann transform $\S_{s,t}$ is defined to be the composition of the isomorphisms so that the following diagram commute:

\begin{displaymath}
    \xymatrix{
        F(\H^\C) \ar@{^{(}->}[rr]^{\delta_{s,t}}  && F((\H\oplus \H)^\C)\\
        L^2(\SC(\H),\tau)  \ar[rr]_{\S_{s,t}} \ar[u]^{A\,\mapsto A\,\Omega}       && L^2(\EX_{\text{hol}}^{s,t}(\H),\tau) \ar[u]_{A\,\mapsto A\,\Omega}}
\end{displaymath}
\end{definition}

When $s=1$, $t=1/2$, the definition coincides with the free Segal-Bargmann transform defined by Biane \cite{Biane1997b}. The free $(s,t)$-Segal-Bargmann transform then maps the $Q_{s,0}^{(k)}(\sqrt{s}X(h))$ to $Q_{s,t}^{(k)}(Z_{s,t}(h))$ because $X=Z_{s,0}$. For a discussion on the left side of the commuting diagram which defined the free Segal-Bargmann transform, see Section \ref{SemiCircularSection}.

We take $\H = L^2(\R)$. $\{X_r := X(\1_{[0,r]})\}_{r\geq 0}$ and $\{Z_{s,t}(r) := Z_{s,t}(\1_{[0,r]})\}_{r\geq 0}$ are concrete constructions of free semicircular Brownian motion and free elliptic $(s,t)$-Brownian motion on Fock spaces. We call a process $F_r$ an adapted semi-circular (resp. elliptic $(s,t)$) process if $F_r\in \SC(L^2([0,r]))$ (resp. $F_r\in \EX^{s,t}(L^2([0,r]))$) for all $r\geq 0$. The free $(s,t)$-Segal-Bargmann transform relates the free stochastic integrals of adapted semi-circular processes and adapted elliptic $(s,t)$ processes nicely.

\begin{proposition}
\label{SegalIntegral}
Suppose that $F_r, G_r$ are adapted semi-circular processes. Then we have
$$\S_{s,t}\left(\int_0^R F_r\;d(\sqrt{s}X_r)\;G_r\right) = \int_0^R \S_{s,t}(F_r)\;dZ_{s,t}(r)\;\S_{s,t}(G_r)$$
where $\S_{s,t}$ is the free Segal-Bargmann transform defined in Definition \ref{FreeSegalBargmann}.
\end{proposition}
\begin{proof}
For an adapted biprocess $A\otimes B\1_{[s_1,s_2]}$, $A, B\in \mathscr{W}_{s_1}:=W^*\{X(\1_{[0,r]}): r\leq s_1\}$. Also, $X_{s_2}-X_{s_1} = X(\1_{[s_1,s_2]})$ with $\1_{[s_1,s_2]}$ is orthogonal to the subspace $L^2([0, s_1])$. In the case that $A = Q_{s,0}^{(k_1)}(\sqrt{s}X(h))$, $B=Q_{s,0}^{(k_2)}(\sqrt{s}X(k))$ are Tchebycheff polynomials at time $s$ of $\sqrt{s}X(h), \sqrt{s}X(h)\in \mathscr{W}_{s_1}$ respectively, it is obvious that, by Proposition \ref{SemiCircularEvaluation} and Proposition \ref{EllipticEvaluation},
$$\S_{s,t}\left(\int_0^R F_r\;dX_r\;G_r\right) = Q_{s,t}^{(k_1)}Z_{s,t}(\1_{[s_1,s_2]})Q_{s,t}^{(k_2)} = \int_0^R \S_{s,t}(F_r)\;dZ_{s,t}(r)\;\S_{s,t}(G_r).$$

For general $A$, $B$, we can approximate $A, B$ by operators of the form $Q_{s,0}^{(k)}(\sqrt{s}X(h))$. And the conclusion holds for processes of the form $A\otimes B\1_{[s_1,s_2]}$. Standard approximation argument completes the proof.
\end{proof}

\subsection{A Biane-Gross-Malliavin Type Theorem}
Gross and Malliavin showed how to derive the Segal-Bargmann-Hall transform on a compact Lie group from the infinite-dimensional Segal-Bargmann transform on the corresponding Lie algebra by the endpoint evaluation maps.
\begin{theorem}[\cite{GrossMalliavin1996}]
Let $K$ be a connected, simply-connected Lie group of compact type and $G$ its complexification. Also let $W(\mathfrak{k})$ be the Wiener space of the Lie algebra $\mathfrak{k}$ of $K$ and $\mathcal{H}(H(\mathfrak{g}))$ be the Hilbert space completion of Gaussian $L^2$-cylinder functions on the space $H(\mathfrak{g}) = \{z\in C([0,1]; \mathfrak{g}): z \text{ is absolutely continuous}, z(0) = 0, \text{and } \|z\|^2=\int_0^1 |z'(r)|\;dr<\infty\}$. Then the following diagram of isometric transforms commute:
\begin{displaymath}
    \xymatrix{
        L^2(K, \rho_1) \ar[rr]^{S_1}  && L_{\text{hol}}^2 (G, \rho_1)  \\
       L^2(W(\mathfrak{k}))  \ar[u]^{\tilde{e}}\ar[rr]^{\S_1}        && \mathcal{H}(H(\mathfrak{g}))\ar[u]_{e}.
       }
\end{displaymath}
where $S_1$ is the Segal-Bargmann-Hall transform on $K$, $\S_1$ is the infinite-dimensional Segal-Bargmann transform and $\tilde{e}$ and $e$ are the endpoint evaluation maps.
\end{theorem}

Biane proved a free version of the Gross-Malliavin Theorem in \cite{Biane1997b}, which is the one-parameter version of of Theorem \ref{GrossMalliavin}. He stated that the one-parameter free unitary Segal-Bargmann transform, denoted $\G_{t,t}$ in this paper, can be recovered by the free Segal-Bargmann transform, which is defined on semi-circular systems, the free analogue of the Wiener space. Instead of evaluating the stochastic path at the endpoint, the Biane applied functional calculus to the $L^2$ function. As in the case of Gross-Malliavin Theorem, all the maps, including the functional calculus map, are unitary isomorphisms.

In this section, we will prove a Biane-Gross-Malliavin type theorem. We rescale the (additive) free Brownian motion to get a time-rescaled free unitary Brownian motion $(u_s(r))_{r\geq 0}$ given by the free stochastic differential equation
$$d u_s(r) = i \sqrt{s}u_s(r) \,dX(r) - \frac{s}{2}u_s(r)\,dr$$
and recall $(b_{s,t}(r))_{r\geq 0}$ is the free $(s,t)$-multiplicative Brownian motion defined as the solution of the free stochastic differential equation
$$db_{s,t}(r) = i b_{s,t}(r) dZ_{s,t}(r)-\frac{1}{2}(s-t)b_{s,t}(r)\,dr$$
with $u_s(0) = b_{s,t}(0) = 1$. We note that $u_s(r) = b_{s,0}(r)$. 
\begin{theorem}
\label{GrossMalliavin}
Let $s>\frac{t}{2}>0$. The following statements hold:
\begin{enumerate}
\item The holomorphic functional calculus, abusedly denoted as $b_{s,t}: F\mapsto F(b_{s,t})$, extends to an isometry from $\A_{s,t}$ onto $L_{\text{hol}}^2(b_{s,t})$.
\item The free Segal-Bargmann transform $\S_{s,t}$ maps $L^2(u_s,\tau)$ onto $L_{\text{hol}}^2(b_{s,t},\tau)$.
\item The following diagram of Segal-Bargmann transforms and functional calculus commute:
\begin{displaymath}
    \xymatrix{
        L^2(\nu_s) \ar[rr]^{u_{s}(1)} \ar[d]_{\G_{s,t}}  && L^2(u_s(1),\tau) \ar[d]^{\S_{s,t}} \\
       \A_{s,t}  \ar[rr]_{b_{s,t}(1)}        && L_{\text{hol}}^2(b_{s,t}(1),\tau).
       }
\end{displaymath}
All maps are unitary isomorphisms.
\end{enumerate}
\end{theorem}

We introduce several lemmas before we prove the theorem.

\begin{lemma}
Let $a_{s,t}(r) = e^{\frac{1}{2}(s-t)r}b_{s,t}(r)$. Then for $n\in\N$
$$d(a_{s,t}(r)^n) = i\sum_{k=1}^n a_{s,t}(r)^k\,dZ_{s,t}(r)\,a_{s,t}(r)^{n-k}+(s-t)\1_{n\geq 2} \sum_{k=1}^{n-1} k a_{s,t}(r)^k \tau(a_{s,t}(r)^{n-k})\,dr.$$
\end{lemma}
\begin{proof}
This is exactly \cite[Proposition 4.4]{Kemp2015}; it is a straightforward calculation in free It\^{o} calculus.
\end{proof}
We recall there are polynomials defined implicitly by the generating functions
\begin{align}
\label{PstGen}
\sum_{n=1}^\infty z^n P_{s,t}^{(n)}(u) = \frac{f_{s,t}(z)u}{1-f_{s,t}(z)u}
\end{align}
where we recall $f_{s,t} = f_s\circ \chi_{s-t}$; $f_s(z) = z e^{\frac{s}{2}\frac{1+z}{1-z}}$ and $\chi_s$ its right inverse defined on $\bbD$.

\begin{lemma}
\label{pstAndbst}
We have
$$\S_{s,t}(P_{sr,tr}^{(1)}(u_s(r))) = b_{s,t}(r)$$
for all $r > 0$.
\end{lemma}
\begin{proof}
Put $v_r = e^{\frac{sr}{2}}u_s(r)$ and $a_{s,t}(r) = e^{\frac{(s-t)r}{2}}b_{s,t}(r)$. Applying the It\^{o} product rule, we have 
$$dv_r = i v_r\:d(\sqrt{s}x_r)$$
and
$$da_{s,t}(r) = i a_{s,t}(r)\:dZ_{s,t}(r).$$
They are free stochastic differential equations whose coefficients are linear polynomials. Because $a_{s,t}(r)$ and $u_r$ satisfy the same initial condition, by Proposition \ref{SegalIntegral} and the preliminary results in Section \ref{FreeStochasticSection}, we have unique solution to the equation and $\S_{s,t}v_r = a_{s,t}(r)$ which is equivalent to saying that
$$\S_{s,t}(e^{\frac{tr}{2}}u_s(r)) = b_{s,t}(r).$$

Differentiating the Equation \eqref{PstGen} with respect to $z$ gives us
$$\sum_{n=1}^\infty n z^{n-1} P_{s,t}^{(n)}(u) = \frac{f_{s,t}'(z)u}{(1-f_{s,t}(z)u)^2}.$$
Since $f_r'(0) = e^{\frac{r}{2}}$, we have $f_{sr,tr}'(0) = f_{sr}'(\chi_{(s-t)r}(0))\chi_{(s-t)r}'(0) = e^{\frac{sr}{2}}e^{-\frac{(s-t)r}{2}} = e^{\frac{tr}{2}}$. Therefore, $P_{sr,tr}^{(1)}(u_s(r)) = e^{\frac{tr}{2}}u_s(r)$ and concludes the result.
\end{proof}

\begin{lemma}
\label{PstRecurrence}
Fix $0<\theta<2$ and $R>0$. Then, there is an open neighborhood $O$ of $0$ such that for all $z\in O$, we have the free stochastic differential equation
$$d\left(\sum_{n=1}^\infty f_{(s-t)r}(z)^n P_{sr, tr}^{(n)}(u_{s}(r))\right) =i \left(\sum_{n=1}^\infty f_{(s-t)r}(z)^n P_{sr,tr}^{(n)}(u_{s}(r))\right)\;d(\sqrt{s}x_{r})\left(\sum_{n=0}^\infty f_{(s-t)r}(z)^n P_{sr,tr}^{(n)}(u_{s}(r))\right),$$
for $0<r<R$.
\end{lemma}
\begin{proof}
First we note that by It\^{o} product rule, we have
\begin{equation}
\label{NewVIto}
d(v_r^n) = i\sum_{k=1}^n v_r^k (d\sqrt{s}x_r) v_r^{n-k} - s\sum_{k=1}^{n-1}k v_r^k \tau(v_r^{n-k})\;dr.
\end{equation}
Take $O$ to be the open neighborhood of $0$ such that Equation \eqref{PstGen} converges in $f_{(s-t)R}(O)$. Observe that 
$$\sum_{n=1}^\infty f_{(s-t)r}(z)^n P_{sr,tr}^{(n)}(u_{s}(r)) = \frac{f_{sr}(z)u_{s}(r)}{1-f_{sr}(z)u_{s}(r)}=\sum_{n=1}^\infty f_{sr}(z)^n u_{s}(r)^n$$
which satisfies the free stochastic integral (by easily replacing the details in \cite{Biane1997b} with \eqref{NewVIto})
$$d\left(\sum_{n=1}^\infty f_{sr}(z)^n u_{s}(r)^n\right) = i \left(\sum_{n=1}^\infty f_{sr}(z)^n u_{s}(r)^n\right)\;d(\sqrt{s}x_{r})\left(\sum_{n=0}^\infty f_{sr}(z)^n u_{s}(r)^n\right).$$
It follows that
$$d\left(\sum_{n=1}^\infty f_{(s-t)r}(z)^n P_{sr, tr}^{(n)}(u_{s}(r))\right) =i \left(\sum_{n=1}^\infty f_{(s-t)r}(z)^n P_{sr,tr}^{(n)}(u_{s}(r))\right)\;d(\sqrt{s}x_{r})\left(\sum_{n=0}^\infty f_{(s-t)r}(z)^n P_{sr,tr}^{(n)}(u_{s}(r))\right).$$\\
\end{proof}

\begin{lemma}
\label{bstReburrence}
Fix $0<\frac{t}{2}<s$ and $R>0$. Then, there is an open neighborhood $O$ of $0$ such that for all $z\in O$, we have the free stochastic differential equation
$$d\left(\sum_{n=1}^\infty f_{(s-t)r}(z)^n b_{s,t}(r)^n\right)=i\left(\sum_{n=1}^\infty f_{(s-t)r}(z)^n b_{s,t}(r)^n\right)\;dZ_{s,t}(r)\;\left(\sum_{n=0}^\infty f_{(s-t)r}(z)^n b_{s,t}(r)^n\right)$$
for $0<r<R$.
\end{lemma}
\begin{proof}
We take $O$ is as in Lemma \ref{PstRecurrence}. Let $h_r (z) = e^{-\frac{(s-t)r}{2}}f_{(s-t)r}(z)$. Applying It\^{o}'s Formula to $h_r(z)^n a_{s,t}(r)^n$, we get
$$d(h_{r}(z)^n a_{s,t}(r)^n) = i\sum_{k=1}^n h_{r}(z)^n a_{s,t}(r)^k \; dZ_{s,t}(r)\;a_{s,t}(r)^{n-k}.$$
Summing over $n$, we get
$$d\left(\sum_{n=1}^\infty h_{r}(z)^n a_{s,t}(r)^n\right)=i\left(\sum_{n=1}^\infty h_{r}(z)^n a_{s,t}(r)^n\right)\;dZ_{s,t}(r)\;\left(\sum_{n=0}^\infty h_{r}(z)^n a_{s,t}(r)^n\right).$$
Now the proposition follows from the fact that $h_{r}(z)^n a_{s,t}(r)^n = f_{(s-t)r}(z)^n b_{s,t}(r)^n$.
\end{proof}

\begin{proposition}
\label{CommuteDiagramPrepared}
We have 
$$\S_{s,t}( P_{sr,tr}^{(n)}(u_{s}(r))) = b_{s,t}(r)^n\;\;\text{and}\;\;\S_{s,t}( P_{sr,tr}^{(n)}(u_{s}(r)^*)) = b_{s,t}(r)^{-n}.$$
\end{proposition}
\begin{proof}
Combining Lemmas \ref{PstRecurrence}, \ref{bstReburrence} and Proposition \ref{SegalIntegral}, we see that $\sum_{n=1}^\infty f_{(s-t)r}(z)^n \S_{s,t}^{-1}(b_{s,t}(r)^n)$ and $\sum_{n=1}^\infty f_{(s-t)r}(z)^n P_{sr,tr}^{(n)}(u_{s}(r))$ satisfy the same free stochastic differential equation. Notice that Lemma \ref{PstRecurrence} is equivalent to
\begin{align}
\label{PowerSeriesForm}
d\left(\sum_{n=1}^\infty f_{(s-t)r}(z)^n P_{sr,tr}^{(n)}(u_{s}(r))\right) = i \sum_{n=1}^\infty \sum_{k=1}^n  f_{(s-t)r}(z)^n P_{sr,tr}^{(k)}(u_{s}(r)) \; d(\sqrt{s}x_r) P_{sr,tr}^{(n-k)}(u_{s}(r))
\end{align}
and the similar equation holds for $d\left(\sum_{n=1}^\infty f_{(s-t)r}(z)^n \S_{s,t}^{-1}(b_{s,t}(r)^n)\right)$. Since $f_{(s-t)r}(z)$ is analytic in $z$ and $f_{(s-t)r}(0)=0$ for all $r>0$, differentiating Equation \eqref{PowerSeriesForm} $n$ times gives us a recurrence relation of $P_{sr,tr}^{(n)}(u_s(r))$ in terms of $P_{sr,tr}^{(k)}(u_s(r))$, $k=1,2,\ldots,n-1$. However, because $d\left(\sum_{n=1}^\infty f_{(s-t)r}(z)^n \S_{s,t}^{-1}(b_{s,t}(r)^n)\right)$ satisfies the same free stochastic differential equation as Equation \eqref{PowerSeriesForm}, the recurrence relation of $P_{sr,tr}^{(n)}(u_s(r))$ holds for $\S_{s,t}^{-1}(b_{s,t}(r)^n)$. Now Lemma \ref{pstAndbst} tells us $\S_{s,t}^{-1}(b_{s,t}(r))=P_{sr,tr}^{(1)}(u_s(r))$ so for $n\geq 2$, 
$$\S_{s,t}( P_{sr,tr}^{(n)}(u_{s}(r))) = b_{s,t}(r)^n$$
follows from the recurrence relation.

For $\S_{s,t}( P_{sr,tr}^{(n)}(u_{s}(r)^*)) = b_{s,t}(r)^{-n}$, we can handle it similarly. Since $f_{s,t}$ maps $\R$ into $\R$, the derivatives are all real; as a result, all the polynomials $P_{sr,tr}^{(n)}$ are of real coefficients. Taking adjoint in Lemma \ref{PstRecurrence}, we have
$$d\left(\sum_{n=1}^\infty f_{(s-t)r}(z)^n P_{sr,tr}^{(n)}(u_{s}(r)^*)\right) =-i \left(\sum_{n=0}^\infty f_{(s-t)r}(z)^n P_{sr,tr}^{(n)}(u_{s}(r)^*)\right)\;d(\sqrt{s}x_{r})\left(\sum_{n=1}^\infty f_{(s-t)r}(z)^n P_{sr,tr}^{(n)}(u_{s}(r)^*)\right).$$
By \cite[Proposition 4.17]{Kemp2015}, $d(b_{s,t}(r)^{-1}) = - i dZ_{s,t}(r) b_{s,t}(r)^{-1} - \frac{1}{2}(s-t)b_{s,t}(r)^{-1}\;dr$ and it is easy to compute that
$$d\left(\sum_{n=1}^\infty f_{(s-t)r}(z)^n b_{s,t}(r)^n\right)=-i\left(\sum_{n=0}^\infty f_{(s-t)r}(z)^n b_{s,t}(r)^n\right)\;dZ_{s,t}(r)\;\left(\sum_{n=1}^\infty f_{(s-t)r}(z)^n b_{s,t}(r)^n\right).$$
Now, it is trivial to proceed as in the preceding paragraph to complete the proof.
\end{proof}
We are finally ready to prove the Biane-Gross-Malliavin type theorem.
\begin{proof}[Proof of Theorem \ref{GrossMalliavin}]
The diagram holds for polynomials by Proposition \ref{CommuteDiagramPrepared}, in which we apply $\theta=\frac{t}{s}$ and $r=s$, and Lemma \ref{GenFunCoincide}. Since polynomials are dense in $L^2(\nu_s)$ and all the maps are unitary isomorphisms, the holomorphic functional calculus map $b_{s,t}$ can be extended to the entire Hilbert space $\A_{s,t}$ and the theorem is established.
\end{proof}

\section*{Acknowledgments}
\addcontentsline{toc}{section}{Acknowledgements}
 I would like to to thank my PhD advisor Todd Kemp for his valuable advice, comments, and for suggesting the problem addressed in the present paper; he carefully read the paper and recommended changes to the original manuscript to help make improvements. I also wish to thank Guillaume C\'{e}bron for useful conversations, especially on Section \ref{ConditionalExpectationSection}. Finally I wish to thank the anonymous referee who gave a lot of useful comments, and made suggestions on how to improve the manuscript.

\newpage

\begin{figure*}
 
\begin{subfigure}{0.5\textwidth}
\includegraphics[scale=.5]{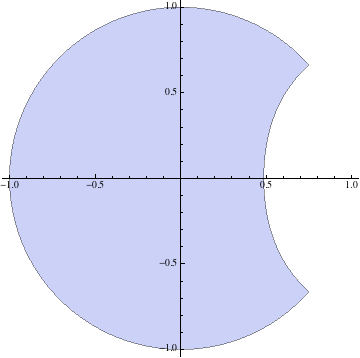} 
\caption{$t=0.5$}
\end{subfigure}
\begin{subfigure}{0.5\textwidth}
\includegraphics[scale=.5]{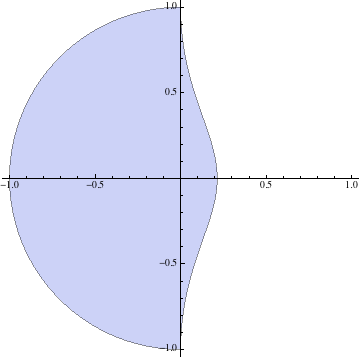}
\caption{$t=2$}
\end{subfigure}
\begin{subfigure}{0.5\textwidth}
\includegraphics[scale=.5]{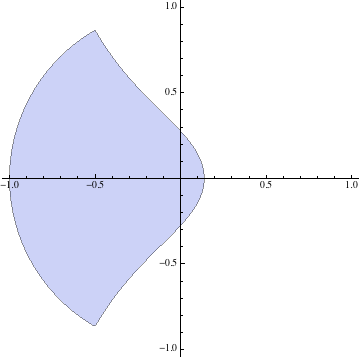}
\caption{$t=3$}
\end{subfigure}
\begin{subfigure}{0.5\textwidth}
\includegraphics[scale=.5]{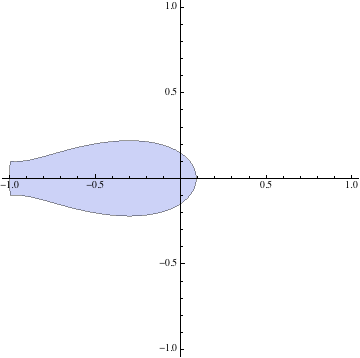}
\caption{$t=3.99$}
\end{subfigure}
\begin{subfigure}{0.5\textwidth}
\includegraphics[scale=.5]{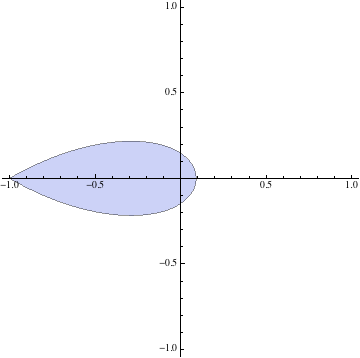}
\caption{$t=4$}
\end{subfigure}
\begin{subfigure}{0.5\textwidth}
\includegraphics[scale=.5]{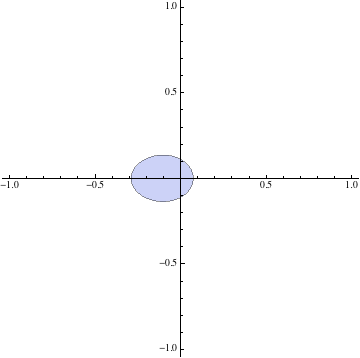}
\caption{$t=4.5$}
\end{subfigure}
 
\caption{The Region $\Omega_t$}
\label{OmegaFigures}
\end{figure*}

\newpage

\begin{figure*}

\begin{subfigure}{0.5\textwidth}
\includegraphics[scale=.22]{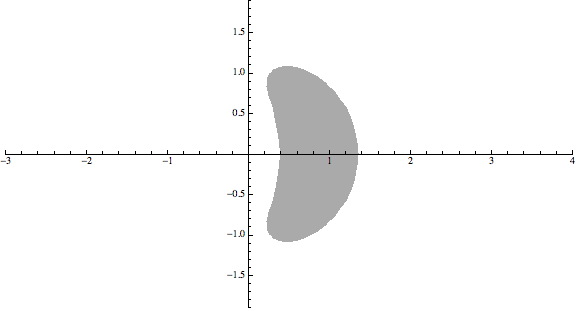} 
\caption{$s=0.5,t=0.1$}
\end{subfigure}
\begin{subfigure}{0.5\textwidth}
\includegraphics[scale=.22]{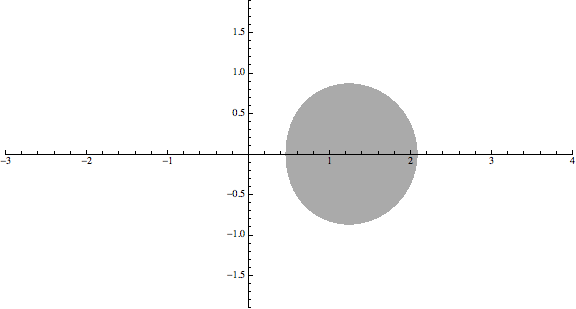}
\caption{$s=0.5,t=0.5$}
\end{subfigure}
\begin{subfigure}{0.5\textwidth}
\includegraphics[scale=.22]{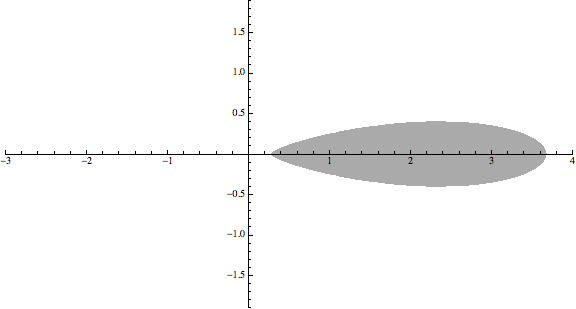}
\caption{$s=0.5,t=0.9$}
\end{subfigure}
\begin{subfigure}{0.5\textwidth}
\includegraphics[scale=.22]{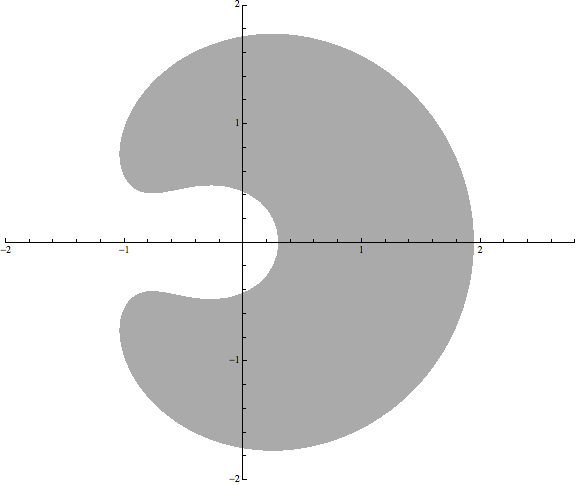}
\caption{$s=3,t=1$}
\end{subfigure}
\begin{subfigure}{0.5\textwidth}
\includegraphics[scale=.22]{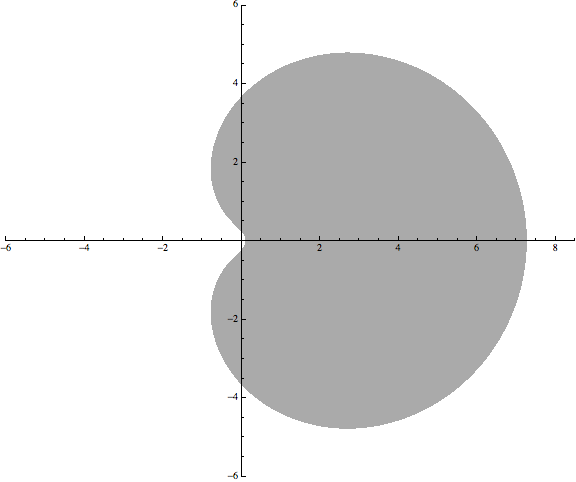}
\caption{$s=3,t=3$}
\end{subfigure}
\begin{subfigure}{0.5\textwidth}
\includegraphics[scale=.22]{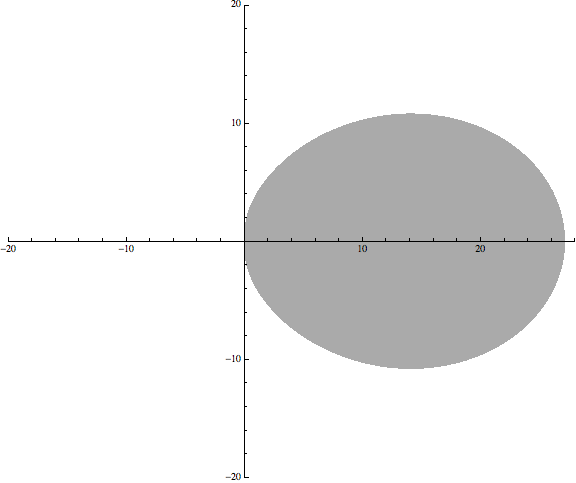}
\caption{$s=3,t=5$}
\end{subfigure}
\begin{subfigure}{0.5\textwidth}
\includegraphics[scale=.22]{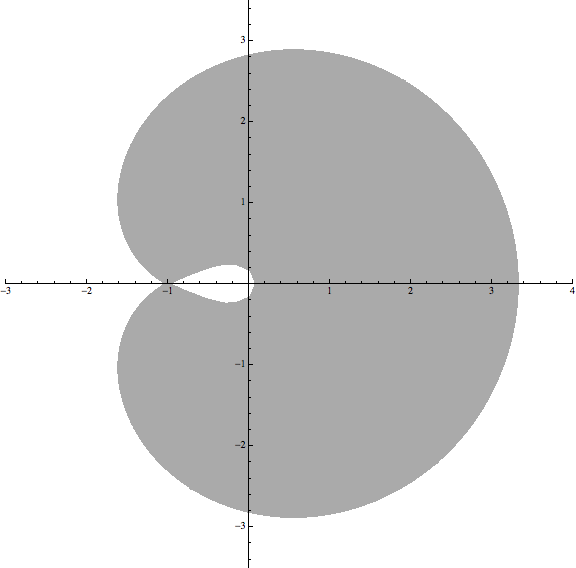}
\caption{$s=4,t=2$}
\end{subfigure}
\begin{subfigure}{0.5\textwidth}
\includegraphics[scale=.22]{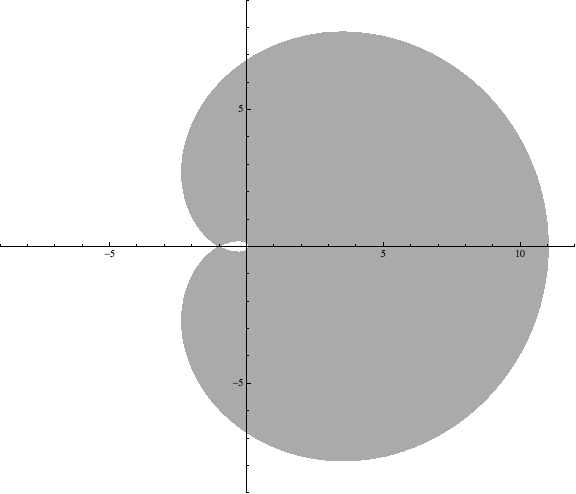}
\caption{$s=4,t=4$}
\end{subfigure}
\begin{subfigure}{0.5\textwidth}
\includegraphics[scale=.22]{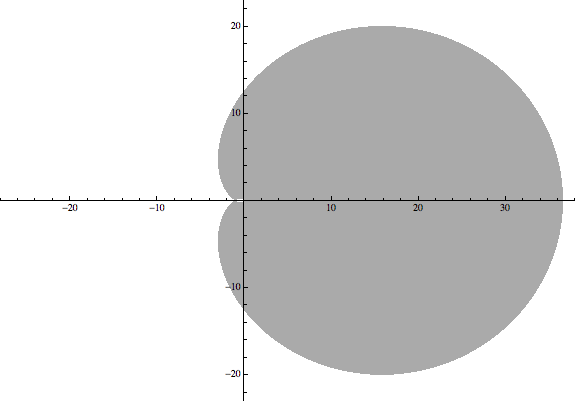}
\caption{$s=4,t=6$}
\end{subfigure} 

\caption{The Region $\Sigma_{s,t}$}
\label{SigmaFigures1}
\end{figure*}

\begin{figure*}
 \begin{subfigure}{0.5\textwidth}
\includegraphics[scale=.3]{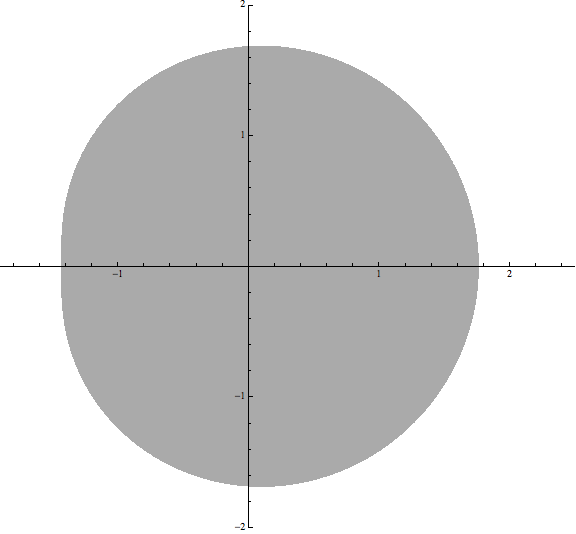}
\caption{$s=5,t=1$}
\end{subfigure}
\begin{subfigure}{0.5\textwidth}
\includegraphics[scale=.3]{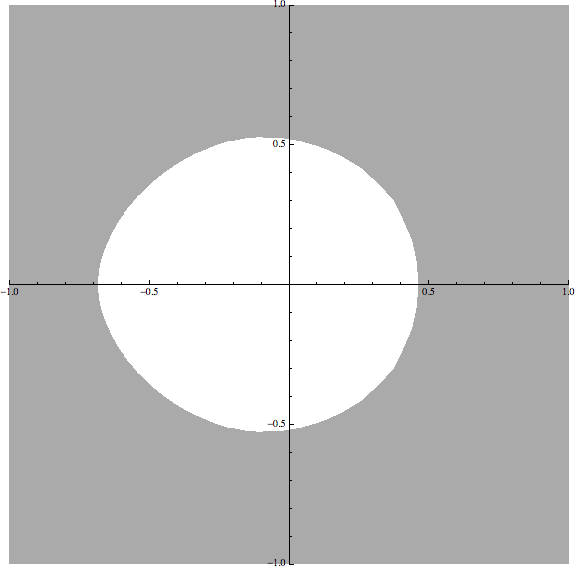}
\caption{$s=5,t=1$ (Close View)}
\end{subfigure}
\begin{subfigure}{0.5\textwidth}
\includegraphics[scale=.3]{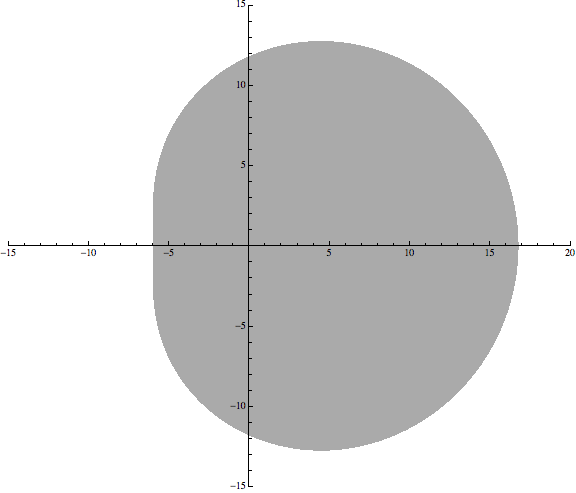}
\caption{$s=5,t=5$}
\end{subfigure}
\begin{subfigure}{0.5\textwidth}
\includegraphics[scale=.3]{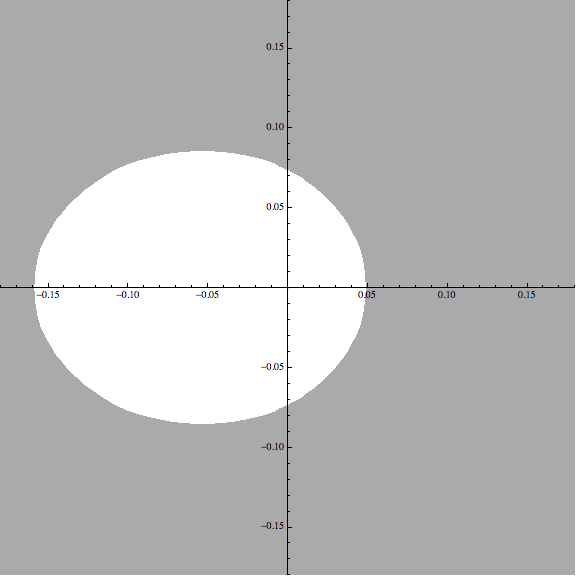}
\caption{$s=5,t=5$ (Close View)}
\end{subfigure}
\begin{subfigure}{0.5\textwidth}
\includegraphics[scale=.3]{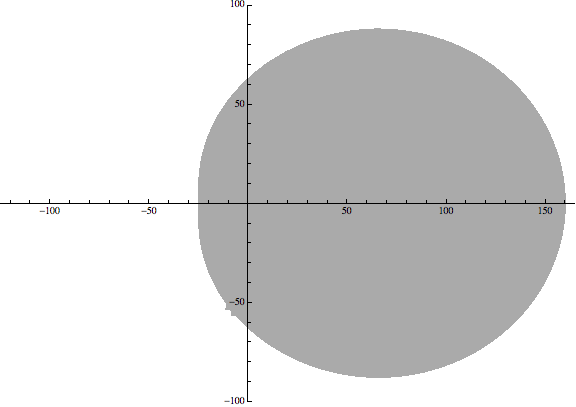}
\caption{$s=5,t=9$}
\end{subfigure}  
\begin{subfigure}{0.5\textwidth}
\includegraphics[scale=.3]{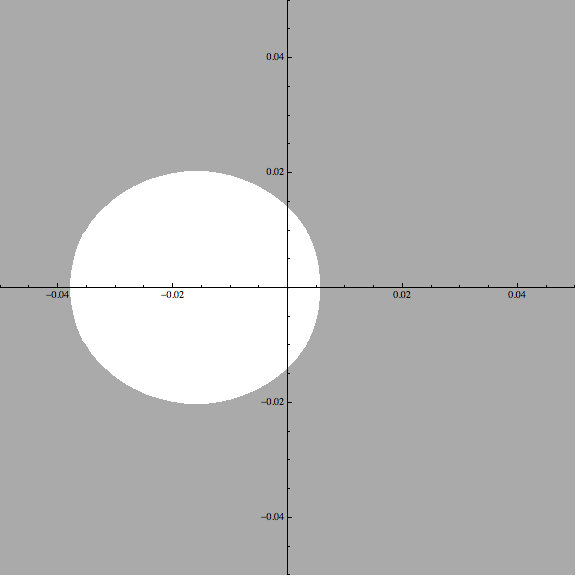}
\caption{$s=5,t=9$ (Close View)}
\end{subfigure}  
 
\caption{The Region $\Sigma_{s,t}$ (Continued)}
\label{SigmaFigures2}
\end{figure*}


\bibliographystyle{acm}
\bibliography{FreeSegalBargmannTransform}

\end{document}